\documentclass[11pt,twoside,reqno]{amsart}

\usepackage{amssymb, latexsym}
\usepackage[italian,english]{babel}
\usepackage{mathrsfs}
\usepackage{amsmath,amsfonts,amsthm}
\usepackage{paralist}
\usepackage[top=2.50cm, bottom=2.50cm, left=2.50cm, right=2.50cm]{geometry}

\usepackage{color}

\numberwithin{equation}{section}

\usepackage{xspace}
\usepackage[bookmarksnumbered,colorlinks]{hyperref}
\usepackage{graphics}
\def\bb#1\eb{\textcolor{blue}
{#1}} %
\def\br#1\er{\textcolor{red}
{#1}} %
\def\bv#1\ev{\textcolor{green}
{#1}} %
\def\bc#1\ec{\textcolor{cyan}
{#1}} %

\usepackage{graphics}

\def\Xint#1{\mathchoice
  {\XXint\displaystyle\textstyle{#1}}%
  {\XXint\textstyle\scriptstyle{#1}}%
  {\XXint\scriptstyle\scriptscriptstyle{#1}}%
  {\XXint\scriptscriptstyle\scriptscriptstyle{#1}}%
  \!\int}
\def\XXint#1#2#3{{\setbox0=\hbox{$#1{#2#3}{\int}$}
  \vcenter{\hbox{$#2#3$}}\kern-.5\wd0}}
\def\-int{\Xint -}

\newcommand{\R}{\mathbb{R}}

\DeclareMathOperator{\X}{\mathbb{H}}

\DeclareMathOperator{\e}{\varepsilon}

\newcommand{\N}{\mathcal{N}}
\DeclareMathOperator{\J}{\mathcal{J}}

\newtheorem{prop}{Proposition}[section]
\newtheorem{lem}{Lemma}[section]
\newtheorem{thm}{Theorem}[section]

\begin{document}
\title[Multiplicity and concentration for fractional Schr\"odinger systems]{Multiplicity and concentration of solutions for fractional  Schr\"odinger systems via penalization method} 

\author[V. Ambrosio]{Vincenzo Ambrosio}
\address{Vincenzo Ambrosio\hfill\break\indent 
Dipartimento di Ingegneria Industriale e Scienze Matematiche \hfill\break\indent
Universit\`a Politecnica delle Marche\hfill\break\indent
Via Brecce Bianche, 12\hfill\break\indent
60131 Ancona (Italy)}
\email{v.ambrosio@univpm.it}


\begin{abstract}
The aim of this paper is to investigate the existence, multiplicity and concentration of positive solutions for the following nonlocal system of fractional Schr\"odinger equations
\begin{equation*}
\left\{
\begin{array}{ll}
\e^{2s} (-\Delta)^{s}u+V(x)u=Q_{u}(u, v) &\mbox{ in } \R^{N},\\
\e^{2s} (-\Delta)^{s}v+W(x)v=Q_{v}(u, v) &\mbox{ in } \R^{N}, \\
u, v>0 &\mbox{ in } \R^{N},
\end{array}
\right.
\end{equation*}
where $\e>0$ is a parameter, $s\in (0, 1)$, $N>2s$, $(-\Delta)^{s}$ is the fractional Laplacian, $V:\R^{N}\rightarrow \R$ and $W:\R^{N}\rightarrow \R$ are positive continuous potentials, $Q$ is a homogeneous $C^{2}$-function with subcritical growth.
In order to relate the number of solutions with the topology of the set where the potentials $V$ and $W$ attain their minimum values, we apply penalization techniques, Nehari manifold arguments and Ljusternik-Schnirelmann theory.
\end{abstract}

\maketitle
\section{Introduction}

\noindent
In this paper we deal with the existence, multiplicity and concentration phenomena of positive solutions for the following nonlinear fractional Schr\"odinger system
\begin{equation}\label{P}
\left\{
\begin{array}{ll}
\e^{2s} (-\Delta)^{s}u+V(x)u=Q_{u}(u, v) &\mbox{ in } \R^{N},\\
\e^{2s} (-\Delta)^{s}v+W(x)v=Q_{v}(u, v)  &\mbox{ in } \R^{N}, \\
u, v>0 &\mbox{ in } \R^{N},
\end{array}
\right.
\end{equation}
where $\e>0$ is a parameter, $s\in (0, 1)$, $N>2s$, $V:\R^{N}\rightarrow \R$ and $W:\R^{N}\rightarrow \R$ are H\"older continuous potentials, $Q$ is a homogeneous $C^{2}$-function with subcritical growth.\\
We assume that there exist a bounded open set $\Lambda\subset \R^{N}$, $x_{0}\in \R^{N}$ and $\rho_{0}>0$ such that:
\begin{compactenum}[$(H1)$]
\item $V(x), W(x)\geq \rho_{0}$ for any $x\in \partial \Lambda$;
\item $V(x_{0}), W(x_{0})< \rho_{0}$;
\item $V(x)\geq V(x_{0})>0$, $W(x)\geq W(x_{0})>0$ for any $x\in \R^{N}$.
\end{compactenum}
Concerning the function $Q: \R^{2}_{+}\rightarrow \R$, where $\R^{2}_{+}=[0, \infty)\times [0, \infty)$, we suppose that $Q\in C^{2}(\R^{2}_{+}, \R)$ satisfies the following conditions:
\begin{compactenum}[$(Q1)$]
\item there exists $p\in (2, 2^{*}_{s})$, with $2^{*}_{s}=\frac{2N}{N-2s}$, such that $Q(t u, tv)=t^{p}Q(u, v)$ for any $t>0$, $(u, v)\in \R^{2}_{+}$;
\item there exists $C>0$ such that $|Q_{u}(u, v)|+|Q_{v}(u, v)|\leq C(u^{p-1}+v^{p-1})$ for any $(u, v)\in \R^{2}_{+}$;
\item $Q_{u}(0, 1)=0=Q_{v}(1, 0)$;
\item $Q_{u}(1, 0)=0=Q_{v}(0, 1)$;
\item $Q(u, v)>0$ for any $u, v>0$;
\item $Q_{u}(u, v), Q_{v}(u, v)\geq 0$ for any $(u, v)\in \R^{2}_{+}$.
\end{compactenum}
Since we are interested in positive solutions of \eqref{P}, we extend the function $Q$ to the whole of $\R^{2}$ by setting $Q(u, v)=0$ if $u\leq 0$ or $v\leq 0$. We note that the $p$-homogeneity of $Q$ implies that the following identity holds:
\begin{equation}\label{2.1}
pQ(u, v)=uQ_{u}(u, v)+vQ_{v}(u, v) \mbox{ for any } (u, v)\in \R^{2},
\end{equation}
and
\begin{equation}\label{2.11}
p(p-1)Q(u, v)=u^{2}Q_{uu}(u, v)+2uvQ_{uv}(u, v)+v^{2}Q_{vv}(u, v) \mbox{ for any } (u, v)\in \R^{2}.
\end{equation}
As a model for $Q$, we can provide the following example given in \cite{DMFS}. 
Let $q\geq 1$ and 
$$
\mathcal{P}_{q}(u, v)=\sum_{\alpha_{i}+\beta_{i}=q} a_{i} u^{\alpha_{i}} v^{\beta_{i}},
$$
where $i\in \{1, \dots, k\}$, $\alpha_{i}, \beta_{i}\geq 1$ and $a_{i}\in \R$. The following functions and their possible combinations, with appropriate choice of the coefficients $a_{i}$, satisfy assumptions $(Q1)$-$(Q5)$ on $Q$
$$
Q_{1}(u, v)=\mathcal{P}_{p}(u, v),  \quad Q_{2}(u, v)=\sqrt[r]{\mathcal{P}_{\ell}(u, v)} \quad \mbox{ and } \quad Q_{3}(u, v)=\frac{\mathcal{P}_{\ell_{1}}(u, v)}{\mathcal{P}_{\ell_{2}}(u, v)},
$$
with $r= \ell p$ and $\ell_{1}-\ell_{2}=p$.
\smallskip

The nonlocal operator $(-\Delta)^{s}$ appearing in \eqref{P}, it is the fractional Laplacian operator which can be defined for any $u: \R^{N}\rightarrow \R$ smooth enough by setting
$$
(-\Delta)^{s}u(x)= C_{N, s}  P.V. \int_{\R^{N}} \frac{(u(x)-u(y))}{|x-y|^{N+2s}} dy  \quad (x\in \R^{N})
$$
where $P. V.$ stands for the Cauchy principal value, and $C_{N, s}$ is a positive
constant depending only on $N$ and $s$; see for instance \cite{DPV, MBRS} for more details. \\
In the scalar case, problem \eqref{P} reduces to the following fractional Schr\"odinger equation
\begin{equation}\label{FSE}
\e^{2s} (-\Delta)^{s}u+V(x)u=f(x, u) \mbox{ in } \R^{N}.
\end{equation}
We recall that a basic motivation to consider (\ref{FSE}) arises in the study of standing wave solutions $\Phi(t, x)=u(x) e^{-\imath c t}$ for the following time-dependent fractional Schr\"odinger equation
\begin{equation}\label{TDFSE}
i  \frac{\partial \Phi}{\partial t} =  (-\Delta)^{s} \Phi +
V(x) \Phi - f(x, |\Phi|) \mbox{ for } (t, x)\in \R\times \R^{N}, 
\end{equation}
which plays a fundamental role in fractional quantum mechanics.
Equation \eqref{TDFSE} was introduced by Laskin \cite{Laskin1, Laskin2} as an extension of the classical nonlinear Schr\"odinger equation \cite{BW, DF, FW, Rab, Wang} in which the Brownian motion of the quantum paths is replaced by a  L\'evy flight.

In the last decade a great attention has been paid to the existence and multiplicity of solutions to \eqref{FSE} under several assumptions on the potential $V(x)$, and involving nonlinearities $f(x, u)$ with subcritical or critical growth.
Felmer et al. \cite{FQT} investigated existence, regularity and qualitative properties of positive solution to \eqref{FSE} when $V=1$ and $f$ is a superlinear function with subcritical growth and satisfying the Ambrosetti-Rabinowitz condition. 
D\'avila et al. \cite{DDPW} used Lyapunov-Schmidt reduction method to prove that \eqref{FSE} has a multi-peak solution when the potential $V\in C^{1, \alpha}(\R^{N})\cap L^{\infty}(\R^{N})$, $\inf_{x\in\R^{N}} V(x)>0$ and $f(x,u)=|u|^{p-1}u$.
Fall et al. \cite{FMV} showed that the concentration points of the solutions of \eqref{FSE} must be the critical points for $V$, as $\e$ tends to zero.
Dipierro et al. \cite{DMV} proved  some existence results to \eqref{FSE} with $V=0$, $f(x, u)=\e h u^{q}+u^{2^{*}_{s}-1}$, where $q\in (0, 1)$ and $h\in L^{1}(\R^{N})\cap L^{\infty}(\R^{N})$, via Concentration-Compactness Principle and mountain pass arguments.
Alves and Miyagaki in \cite{AM} (see also \cite{A3}) used the extension method \cite{CS} and the penalization technique in \cite{DF} to investigate the existence and concentration of positive solutions to (\ref{FSE}) when $f$ is a continuous function having a subcritical growth, and the potential $V$ is a continuous function having a local minimum. 
Further results related to \eqref{FSE} can be found in \cite{AlAm, Ambrosio, AI2, FS, FLS, Secchi1} in which the authors established several existence and multiplicity results by using appropriate and different variational and topological methods. \\
In this paper we focus our attention on the multiplicity and concentration of positive solutions for fractional Schr\"odinger systems.\\
We recall that in the classical literature, many interesting papers \cite{Alves, AFF, AFF2, AS, AY, BS, FF} considered the existence, multiplicity and symmetry of solutions for elliptic systems of the type
\begin{equation}\label{CSS}
\left\{
\begin{array}{ll}
-\e^{2} \Delta u+V(x)u=G_{u}(u, v)  &\mbox{ in } \R^{N},\\
-\e^{2} \Delta v+W(x)v=G_{v}(u, v)  &\mbox{ in } \R^{N}, \\
u, v>0  &\mbox{ in } \R^{N}.
\end{array}
\right.
\end{equation}
In particular way, in \cite{Alves, AFF, AFF2}, the authors investigated positive solutions to \eqref{CSS}, via a suitable variant of the penalization method introduced by del Pino and Felmer in \cite{DF} to study a class of nonlinear Schr\"odinger equations. \\ 
Differently from the local case, in the fractional context there are only few papers \cite{ASy, choi, lm, ww} dealing with fractional systems in $\R^{N}$, and, as far as we know, no results on the multiplicity and concentration of solutions for fractional nonlinear Schr\"odinger systems are available.\\
The goal of this work is to give a first result in this direction, generalizing the multiplicity and concentration results in \cite{AFF} for nonlocal system \eqref{P}.\\
Before stating our results, we need to introduce some notations.
Fix $\xi\in \R^{N}$, and we consider the following autonomous system
\begin{equation*}
\left\{
\begin{array}{ll}
 (-\Delta)^{s}u+V(\xi)u=Q_{u}(u, v)  &\mbox{ in } \R^{N},\\
 (-\Delta)^{s}v+W(\xi)v=Q_{v}(u, v) &\mbox{ in } \R^{N}, \\
u, v>0 &\mbox{ in } \R^{N}.
\end{array}
\right.
\end{equation*}
Let $\J_{\xi}: H^{s}(\R^{N})\times H^{s}(\R^{N})\rightarrow \R$ be the Euler-Lagrange functional associated with the above problem, i.e. 
$$
\J_{\xi}(u, v)=\frac{1}{2}\|(u, v)\|^{2}_{\xi}-\int_{\R^{N}} Q(u, v) dx, 
$$
where
$$
\|(u, v)\|^{2}_{\xi}:=\int_{\R^{N}} |(-\Delta)^{\frac{s}{2}} u|^{2}+ |(-\Delta)^{\frac{s}{2}} v|^{2} dx+\int_{\R^{N}} (V(\xi) u^{2}+W(\xi) v^{2}) dx. 
$$
As in \cite{ASy}, we can see that assumptions $(H3)$, $(Q1)$ and $(Q2)$, show that $\J_{\xi}$ possesses a mountain pass geometry, so we can consider the mountain pass value
$$
C(\xi):=\inf_{\gamma\in \Gamma} \max_{t\in [0, 1]} \J_{\xi}(\gamma(t))
$$
where 
$$
\Gamma:=\{\gamma\in C([0, 1], \X_{0}): \gamma(0)=0, \J_{\xi}(\gamma(1))\leq 0\}.
$$
Moreover, we can prove (see Section $2$) that $\xi \mapsto C(\xi)$ is a continuous function and that $C(\xi)$ can be also characterized as
$$
C(\xi)=\inf_{(u, v)\in \mathcal{N}_{\xi}} \J_{\xi}(u, v),
$$
where $\mathcal{N}_{\xi}$ is the Nehari manifold associated with $\J_{\xi}$. 
From the results in \cite{ASy}, we know that, for any fixed $\xi\in \R^{N}$, $C(\xi)$ is achieved and in view of condition $(H3)$ we can deduce that $C(x_{0})\leq C(\xi)$ for any $\xi\in \R^{N}$, which yields
$$
M:=\left\{x\in \R^{N}: C(x)=\inf_{\xi\in \R^{N}} C(\xi)\right\}\neq \emptyset.
$$
We recall that if $Y$ is a given closed set of a topological space $X$, we denote by $cat_{X}(Y)$ the Ljusternik-Schnirelmann category of $Y$ in $X$, that is the least number of closed and contractible sets in $X$ which cover $Y$.
\smallskip

\noindent
With the above notations, the statement of our main result is the following one.
\begin{thm}\label{thm1}
Assume that $(H1)$-$(H3)$ and $(Q1)$-$(Q6)$ hold. Then, for any $\delta>0$ satisfying
$$
M_{\delta}=\{x\in \R^{N}: dist(x, M)\leq\delta\}\subset \Lambda,
$$ 
there exists $\e_{\delta}>0$ such that, for any $\e\in (0, \e_{\delta})$, system \eqref{P} admits at least $cat_{M_{\delta}}(M)$ solutions.
Moreover, if $(u_{\e}, v_{\e})$ is a solution to \eqref{P} and $P_{\e}$ and $Q_{\e}$ are global maximum points of $u_{\e}$ and $v_{\e}$ respectively, then $C(P_{\e}), C(Q_{\e})\rightarrow C(x_{0})$ as $\e\rightarrow 0$, and we have the following estimates: 
\begin{align}\label{DEuv}
u_{\e}(x)\leq \frac{C \e^{N+2s}}{\e^{N+2s}+|x-P_{\e}|^{N+2s}} \,\, \mbox{ and } \,\, v_{\e}(x)\leq \frac{C \e^{N+2s}}{\e^{N+2s}+|x-Q_{\e}|^{N+2s}} \quad \forall x\in \R^{N}. 
\end{align}
\end{thm}
The proof of Theorem \ref{thm1} is obtained by combining in a suitable way some variational arguments inspired by \cite{Alves, AFF} with  some ideas used in \cite{AM, A3} to deal with fractional Schr\"odinger equations.
Firstly, we use the penalization technique introduced by Alves \cite{Alves} modifying appropriately the function $Q(u, v)$ outside the set $\Lambda$. 
In this way, the energy functional $\J_{\e}$  associated with the modified problem satisfies the assumptions of the mountain pass theorem \cite{ambrosetti-rab}, and we can find a nontrivial solution of the modified problem.
Since we are interested in obtaining a multiplicity result for the modified problem, we study the energy functional $\J_{\e}$ restricted to its Nehari Manifold $\N_{\e}$, and we employ a technique introduced by Benci and Cerami in \cite{BC}. 
The main ingredient is to make precisely comparisons between the category of some sublevel sets of the functional $\J_{\e}$ and the category of the set $M$.
Therefore, using Ljusternik-Schnirelmann theory, we obtain the existence of multiple solutions $(u_{\e}, v_{\e})$ for the modified problem. Now, in order to prove that these solutions are also solutions to \eqref{P} provided that $\e>0$ is sufficiently small, we use a different approach from \cite{Alves, AFF}, because the techniques developed for the local case can not be adapted in our context due to the presence of the nonlocal operator $(-\Delta)^{s}$.
More precisely, motivated by \cite{AlAm, AM, Ambrosio, A3}, we use a Moser iteration argument to estimate the $L^{\infty}$-norm of $(u_{\e}, v_{\e})$, and by constructing suitable comparison functions based on the Bessel kernel \cite{FQT}, we are able to show that $|(u_{\e}(x), v_{\e}(x))|\rightarrow 0$ as $|x|\rightarrow \infty$ uniformly in $\e$.  This fact will be fundamental to achieve our aim. Finally, we also study the behavior of the maximum points of solutions to \eqref{P}. \\
We would like to point out that Theorem \ref{thm1} is in clear accordance with the local case, and it can be seen as the nonlocal counterpart of Theorem $1.1$ in \cite{AFF}.\\
We also emphasize that, to our knowledge, this is the first result in which the penalization technique combined with Ljusternik-Schnirelmann theory allows us to obtain multiple solutions for subcritical fractional system \eqref{P}.
\smallskip

\noindent
The paper is organized as follows. In Section $2$ we collect some preliminary facts about the fractional Sobolev spaces and fractional autonomous systems. 
In Section $3$ we introduce the modified problem. 
In Section $4$ we prove some compactness results for the modified functional. In Section $5$ we present the proof of Theorem \ref{thm1}. 

\section{preliminaries and technical results}

In this preliminary section we recall some results concerning the fractional Sobolev spaces and we introduce the functional setting.\\
For any $s\in (0,1)$ we define $\mathcal{D}^{s, 2}(\R^{N})$ as the completion of $C^{\infty}_{0}(\R^{N})$ with respect to
$$
\int_{\R^{N}} |(-\Delta)^{\frac{s}{2}} u|^{2} dx =\frac{C_{N,s}}{2}\iint_{\R^{2N}} \frac{|u(x)-u(y)|^{2}}{|x-y|^{N+2s}} \, dx \, dy,
$$
or equivalently
$$
\mathcal{D}^{s, 2}(\R^{N})=\left\{u\in L^{2^{*}_{s}}(\R^{N}): \int_{\R^{N}} |(-\Delta)^{\frac{s}{2}} u|^{2} dx<\infty\right\}.
$$
Let us introduce the fractional Sobolev space
$$
H^{s}(\R^{N}):= \left\{u\in L^{2}(\R^{N}) : \int_{\R^{N}} |(-\Delta)^{\frac{s}{2}} u|^{2} dx<\infty \right \}
$$
endowed with the natural norm 
$$
\|u\|_{H^{s}(\R^{N})} := \sqrt{\int_{\R^{N}} |(-\Delta)^{\frac{s}{2}} u|^{2} dx + \int_{\R^{N}} |u|^{2} dx}.
$$

\noindent
We recall the following fundamental embeddings:
\begin{thm}\cite{DPV}\label{Sembedding}
Let $s\in (0,1)$ and $N>2s$. Then there exists a sharp constant $S_{*}=S(N, s)>0$
such that for any $u\in H^{s}(\R^{N})$
\begin{equation}\label{FSI}
\left(\int_{\R^{N}} |u|^{2^{*}_{s}} dx\right)^{\frac{2}{2^{*}_{s}}}  \leq S_{*}^{-1} \int_{\R^{N}} |(-\Delta)^{\frac{s}{2}} u|^{2} dx . 
\end{equation}
Moreover, $H^{s}(\R^{N})$ is continuously embedded in $L^{q}(\R^{N})$ for any $q\in [2, 2^{*}_{s}]$ and compactly in $L^{q}_{loc}(\R^{N})$ for any $q\in [1, 2^{*}_{s})$. 
\end{thm}
Now we collect some technical results which will be useful later. 
Fixed $\xi\in \R^{N}$, let us consider the following subcritical autonomous system
\begin{equation}\label{P0}
\left\{
\begin{array}{ll}
 (-\Delta)^{s}u+V(\xi)u=Q_{u}(u, v)  &\mbox{ in } \R^{N},\\
 (-\Delta)^{s}v+W(\xi)v=Q_{v}(u, v) &\mbox{ in } \R^{N}, \\
u, v>0 &\mbox{ in } \R^{N}.
\end{array}
\right.
\end{equation}
We set $\X_{0}=H^{s}(\R^{N})\times H^{s}(\R^{N})$ endowed with the following norm
$$
\|(u, v)\|^{2}_{\xi}:=\int_{\R^{N}} |(-\Delta)^{\frac{s}{2}} u|^{2}+ |(-\Delta)^{\frac{s}{2}} v|^{2} dx+\int_{\R^{N}} (V(\xi) u^{2}+W(\xi) v^{2}) dx. 
$$
Clearly, $\X_{0}$ is a Hilbert space.
Let us introduce the functional $\J_{\xi}: \X_{0}\rightarrow \R$ defined as 
$$
\J_{\xi}(u, v)=\frac{1}{2}\|(u, v)\|^{2}_{\xi}-\int_{\R^{N}} Q(u, v) dx. 
$$
Since $\J_{\xi}$ has a mountain pass geometry (see \cite{ASy}), we can define the minimax level
$$
C(\xi):=\inf_{\gamma\in \Gamma} \max_{t\in [0, 1]} \J_{\xi}(\gamma(t))
$$
where 
$$
\Gamma:=\{\gamma\in C([0, 1], \X_{0}): \gamma(0)=0, \J_{\xi}(\gamma(1))\leq 0\}.
$$
Using Theorem $3.1$ in \cite{ASy}, we know that problem \eqref{P0} admits a weak solution. 

\noindent
Next we give the proof of the following result which plays an important role to study \eqref{P}.
\begin{lem}\label{C0}
The map $\xi\mapsto C(\xi)$ is continuous.
\end{lem}
\begin{proof}
For $\xi\in \R^{N}$, let $\{\zeta_{n}\}, \{\lambda_{n}\}\subset \R^{N}$ be two sequences such that
\begin{compactenum}[$(a)$]
\item $\zeta_{n}\rightarrow \xi$ and $C(\zeta_{n})\geq C(\xi)$ for all $n\in \mathbb{N}$,
\item $\lambda_{n}\rightarrow \xi$ and $C(\lambda_{n})\leq C(\xi)$ for all $n\in \mathbb{N}$.
\end{compactenum}
We aim to prove that $C(\zeta_{n}), C(\lambda_{n})\rightarrow C(\xi)$ as $n\rightarrow \infty$.
Using Theorem $3.1$ in \cite{ASy}, we know that there exists $w=(u, v)\in \X_{0}$ such that 
$$
\J_{\xi}(w)=C(\xi) \mbox{ and } \J_{\xi}'(w)=0.
$$
For any $n\in \mathbb{N}$, let $t_{n}>0$ be such that 
$$
C(\zeta_{n})\leq \J_{\zeta_{n}}(t_{n} w)=\max_{t\geq 0} \J_{\zeta_{n}}(t w).
$$
We can show that $t_{n}\rightarrow 1$. Indeed $\J_{\xi}'(w)=0$ and \eqref{2.1} imply that
$$
\|(u, v)\|_{\xi}^{2}=p\int_{\R^{N}} Q(u, v) dx.
$$
By the definition of $t_{n}>0$ we know that $\frac{d}{dt} \J_{\zeta_{n}}(t u, t v)\mid_{t=t_{n}}=0$, so, using $(Q1)$ and \eqref{2.1}, we get
$$
\|(u, v)\|_{\zeta_{n}}^{2}=pt_{n}^{p-1}\int_{\R^{N}} Q(u, v) dx.
$$
Thus, using the continuity of $V$ and $W$, and the fact that $\zeta_{n}\rightarrow \xi$, we deduce that $t_{n}\rightarrow 1$ as $n\rightarrow \infty$. Moreover, we can see that $\J_{\zeta_{n}}(t_{n} w)\rightarrow \J_{\xi}(w)$ as $n\rightarrow \infty$. 
Therefore
\begin{equation*}
\limsup_{n\rightarrow \infty} C(\zeta_{n})\leq C(\xi).
\end{equation*}
From $(a)$ we can deduce that $\liminf_{n\rightarrow \infty} C(\zeta_{n})\geq C(\xi)$, which implies that $C(\zeta_{n})\rightarrow C(\xi)$ as $n\rightarrow \infty$.
Now we show that $C(\lambda_{n})\rightarrow C(\xi)$ as $n\rightarrow \infty$. Using Theorem $3.1$ in \cite{ASy}, there exists $w_{n}=(u_{n}, v_{n})$ such that 
\begin{equation}\label{AS7}
\J_{\lambda_{n}}(w_{n})=C(\lambda_{n}) \mbox{ and } \J_{\lambda_{n}}'(w_{n})=0.
\end{equation}
Let $p_{n}, q_{n}\in \R^{N}$ be such that 
$$
u_{n}(p_{n})=\max_{x\in \R^{N}} u_{n}(x) \mbox{ and } v_{n}(q_{n})=\max_{x\in \R^{N}} v_{n}(x),
$$
and we set $z_{n}=u_{n}+v_{n}$. By $(Q2)$, there exists $K>0$ such that $z_{n}$ satisfies
$$
(-\Delta)^{s} z_{n}+\alpha z_{n}\leq K z_{n}^{p-1} \mbox{ in } \R^{N},
$$
where $\alpha=\min\{V(x_{0}), W(x_{0})\}$.
If we denote by $z_{n}(r_{n})=\max_{x\in \R^{N}} z_{n}(x)$, we can use the integral representation formula for the fractional Laplacian (see \cite{DPV}) to see that 
$$
(-\Delta)^{s} z_{n}(r_{n})=\frac{C_{N, s}}{2}\int_{\R^{N}} \frac{2z_{n}(r_{n})-z_{n}(r_{n}+x)-z_{n}(r_{n}-x)}{|x|^{N+2s}} dx\geq 0.
$$
Therefore,
$$
0<\alpha \leq K z_{n}(r_{n})^{p-2}\leq K(u_{n}(p_{n})+v_{n}(q_{n}))^{p-2}.
$$
Hence, there exists $\delta=(\frac{\alpha}{K})^{\frac{1}{p-2}}>0$ such that for any $n\in \mathbb{N}$
$$
u_{n}(p_{n})\geq \frac{\delta}{4} \mbox{ or } v_{n}(q_{n})\geq \frac{\delta}{4} .
$$
Consequently, there exists an infinite subset $\mathbb{M}\subset \mathbb{N}$ such that at least one of the following cases occurs:
\begin{compactenum}[$(i)$]
\item $u_{n}(p_{n})\geq \frac{\delta}{4}$ for any $n\in \mathbb{M}$,
\item $v_{n}(q_{n})\geq \frac{\delta}{4}$ for any $n\in \mathbb{M}$.
\end{compactenum} 
Let us assume that $(i)$ occurs, and  define
$$
\hat{u}_{n}(x)=u_{n}(x+p_{n}) \mbox{ and } \hat{v}_{n}(x)=v_{n}(x+p_{n}).
$$
From \eqref{AS7}, we may assume, up to a subsequence, that $\hat{u}_{n}\rightharpoonup \hat{u}$ and $\hat{v}_{n}\rightharpoonup \hat{v}$ in $\X_{0}$. Since $\lambda_{n}\rightarrow \xi$, we can note that the function $\hat{w}=(\hat{u}, \hat{v})$ verifies $\J_{\xi}(\hat{w})=C(\xi)$ and $\J'_{\xi}(\hat{w})=0$.
Using \eqref{AS7}, we can see that $\langle \J_{\lambda_{n}}'(w_{n}), \hat{w}(\cdot-p_{n})\rangle=0$, that is $$
\langle \J_{\lambda_{n}}'(\hat{w}_{n}), \hat{w}\rangle=0.
$$
By the weak convergence in $\X_{0}$ and $\lambda_{n}\rightarrow \xi$ as $n\rightarrow \infty$, we can pass to the limit in the above relation and we find $\langle \J_{\xi}'(\hat{w}), \hat{w}\rangle=0$. 
Now, fix $\theta\in (2, p)$. Using the weak convergence, Fatou's Lemma, \eqref{2.11} and $\langle \J_{\xi}'(\hat{w}), \hat{w}\rangle=0$, we can see that
\begin{align*}
C(\xi)&\leq \J_{\xi}(\hat{w})=\J_{\xi}(\hat{w})-\frac{1}{\theta}\langle \J_{\xi}'(\hat{w}), \hat{w}\rangle \\
&=\left(\frac{1}{2}-\frac{1}{\theta}\right) \|\hat{w}\|^{2}_{\xi}+\left(\frac{p}{\theta}-1\right) \int_{\R^{N}} Q(\hat{u}, \hat{v}) dx \\
&\leq \liminf_{n\rightarrow \infty} \left[\left(\frac{1}{2}-\frac{1}{\theta}\right) \|(\hat{u}_{n}, \hat{v}_{n})\|^{2}_{\lambda_{n}}+\left(\frac{p}{\theta}-1\right) \int_{\R^{N}} Q(\hat{u}_{n}, \hat{v}_{n}) dx\right] \\
&=\liminf_{n\rightarrow \infty} \left[ \J_{\lambda_{n}}(u_{n}, v_{n})-\frac{1}{\theta} \langle \J_{\lambda_{n}}'(u_{n}, v_{n}), (u_{n}, v_{n})\rangle \right] \\
&=\liminf_{n\rightarrow \infty} \J_{\lambda_{n}}(u_{n}, v_{n})=\liminf_{n\rightarrow \infty} C(\lambda_{n}).
\end{align*}
This and condition $(b)$ yields $C(\lambda_{n})\rightarrow C(\xi)$ as $n\rightarrow \infty$.

\end{proof}

Let us note that
$$
C(\xi)=\inf_{(u, v)\in \N_{\xi}} \J_{\xi}(u, v)
$$
where
$$
\N_{\xi}:=\{(u, v)\in \X_{0}\setminus \{(0, 0)\}:  \langle \J'_{\xi}(u, v),(u, v)\rangle=0\}.
$$
Since the minimax level $C(\xi)$ is achieved and using $(H1)$-$(H3)$, we can see that
$$
M=\left\{x\in \R^{N}: C(x)=\inf_{\xi\in \R^{N}} C(\xi)\right\}\neq \emptyset.
$$ 
Now we prove the following fundamental result.
\begin{lem}\label{C00}
$C^{*}:=C(x_{0})=\inf_{\xi\in \Lambda} C(\xi)<\min_{\xi\in \partial \Lambda} C(\xi)$.
\end{lem}
\begin{proof}
Let us denote by $b_{\rho_{0}}$ the minimax level of mountain pass theorem associated with the functional $F_{\rho_{0}}: \X_{0}\rightarrow \R$ given by
\begin{equation*}
F_{\rho_{0}}(u,v)= \frac{1}{2}\int_{\R^{N}} (|(-\Delta)^{\frac{s}{2}} u|^{2}+ |(-\Delta)^{\frac{s}{2}} v|^{2} + \rho_{0}u^{2} + \rho_{0}v^{2})\, dx - \int_{\R^{N}} Q(u,v) \, dx. 
\end{equation*}
Using the definition of $F_{\rho_{0}}$ and $(H1)$, we have for $\xi \in \partial \Lambda$  
\begin{equation*}
F_{\rho_{0}}(u, v) \leq \J_{\xi} (u, v) \, \mbox{ for all } (u, v)\in \X_{0},
\end{equation*}
so we get
\begin{equation*}
b_{\rho_{0}}\leq C(\xi) \, \mbox{ for all }  \xi \in \partial \Lambda.
\end{equation*}
Thus 
\begin{equation}\label{2.1-1}
b_{\rho_{0}}\leq \min_{\xi \in \partial \Lambda} C(\xi). 
\end{equation}
On the other hand, from $(H2)$, we can see that
\begin{equation*}
F_{\rho_{0}}(u, v) > \J_{x_{0}}(u,v)\, \mbox{ for all }  (u,v)\in \X_{0}\setminus \{0\},
\end{equation*}
from where we can conclude that
\begin{equation*}
C(x_{0})<b_{\rho_{0}}
\end{equation*}
which implies 
\begin{equation}\label{2.2-1}
\inf_{\xi \in \Lambda} C(\xi)< b_{\rho_{0}}.
\end{equation}
Putting together \eqref{2.1-1} and \eqref{2.2-1} we obtain
\begin{equation*}
\inf_{\xi \in \Lambda} C(\xi)< \min_{\xi \in \partial \Lambda} C(\xi).
\end{equation*}
This ends the proof of lemma. 
\end{proof}

Finally, we recall  the following compactness property related to minimizing sequences of the autonomous system, whose proof follows the lines of Theorem $3.1$ in \cite{ASy}.
\begin{thm}\cite{ASy}\label{prop2.2}
Let $\{(u_{n}, v_{n})\}\subset \N_{\xi}$ be a sequence such that $\J_{\xi}(u_{n}, v_{n})\rightarrow C^{*}$ and $(u_{n}, v_{n})\rightharpoonup (u, v)$ in $\X_{0}$. Then there exists $\{\tilde{y}_{n}\}\subset \R^{N}$
such that the translated sequence $\{(u_{n}(\cdot+\tilde{y}_{n}), v_{n}(\cdot+\tilde{y}_{n}))\}$ strongly converges to $(\tilde{u}, \tilde{v})\in \N_{\xi}$ with $\J_{\xi}(\tilde{u}, \tilde{v})=C^{*}$.
Moreover, if $(u, v)\neq 0$, then $\{\tilde{y}_{n}\}$ can be taken identically zero and therefore $(u_{n}, v_{n})\rightarrow (u, v)$ in $\X_{0}$.
\end{thm}

\section{the modified problem}
In this section we introduce a penalty function in order to study solutions of problem \eqref{P}.
Firstly, we observe that, by using the change of variable $x\mapsto \e x$, the analysis of \eqref{P} is equivalent to consider the following problem
\begin{equation}\label{P'}
\left\{
\begin{array}{ll}
 (-\Delta)^{s}u+V(\e x)u=Q_{u}(u, v)  &\mbox{ in } \R^{N},\\
 (-\Delta)^{s}v+W(\e x) v=Q_{v}(u, v) &\mbox{ in } \R^{N}, \\
u, v>0 &\mbox{ in } \R^{N}.
\end{array}
\right.
\end{equation}
At this point we choose $a>0$ and $\eta\in C^{2}(\R, \R)$ a non-increasing function such that 
\begin{equation}\label{3}
\eta=1 \mbox{ on } (-\infty, a], \, \eta=0 \mbox{ on } [5a, \infty), \, |\eta'(t)|\leq \frac{C}{a}, \, \mbox{ and } |\eta''(t)|\leq \frac{C}{a^{2}}.
\end{equation} 
Using $\eta$, we introduce the following function $\hat{Q}: \R^{2}\rightarrow \R$ by setting
$$
\hat{Q}(u, v):=\eta(|(u, v)|) Q(u, v)+(1-\eta(|(u, v)|) A(u^{2}+v^{2}),
$$
where 
$$
A:=\max\left\{ \frac{Q(u, v)}{u^{2}+v^{2}}: (u, v)\in \R^{2}, a\leq |(u, v)|\leq 5a \right\}>0.
$$
Let us observe that $A\rightarrow 0$ as $a\rightarrow 0^{+}$, so we may assume that $A<\min\{V(x_{0}), W(x_{0})\}$.
Now we define $H: \R^{N}\times \R^{2}\rightarrow \R$ by setting
$$
H(x, u, v):=\chi_{\Lambda}(x) Q(u, v)+(1-\chi_{\Lambda}(x)) \hat{Q}(u, v).
$$
As in \cite{Alves}, we can prove the following useful properties of the penalized function $H$.
\begin{lem}\label{lem2.2Alves}
The function $H$ satisfies the following estimates
\begin{equation}\label{2.4A}
pH(x, u, v)=u H_{u}(x, u, v)+v H_{v}(x, u, v) \mbox{ for any } x\in \Lambda
\end{equation}
and
\begin{equation}\label{2.5A}
2H(x, u, v)\leq u H_{u}(x, u, v)+v H_{v}(x, u, v) \mbox{ for any } x\in \R^{N}\setminus \Lambda.
\end{equation}
Moreover, for any $k>0$ fixed, we can choose the constant $a>0$ sufficiently small such that 
\begin{equation}\label{2.6A}
u H_{u}(x, u, v)+v H_{v}(x, u, v)\leq \frac{1}{k} (V(x) u^{2}+W(x) v^{2}) \mbox{ for any } x\in \R^{N}\setminus \Lambda
\end{equation}
and
\begin{equation}\label{2.7A}
\frac{|H_{u}(x, u, v)|}{a}, \frac{|H_{v}(x, u, v)|}{a} \leq \frac{\alpha}{4} \mbox{ for any } x\in \R^{N}\setminus \Lambda,
\end{equation}
where $\alpha:=\min\{V(x_{0}), W(x_{0})\}$.
\end{lem}

Now we consider the following modified problem
\begin{equation}\label{P''}
\left\{
\begin{array}{ll}
 (-\Delta)^{s}u+V(\e x)u=H_{u}(\e x, u, v)  &\mbox{ in } \R^{N},\\
 (-\Delta)^{s}v+W(\e x) v=H_{v}(\e x, u, v) &\mbox{ in } \R^{N}, \\
u, v>0 &\mbox{ in } \R^{N}.
\end{array}
\right.
\end{equation}
Then, by the definition of $H$ and $\hat{Q}$, to study solutions of \eqref{P'}, we will look for solutions $(u_{\e}, v_{\e})$ to \eqref{P''} such that 
$$
|(u_{\e}(x), v_{\e}(x))|\leq a \mbox{ for any } x\in \R^{N}\setminus \Lambda_{\e},
$$
where $\Lambda_{\e}:=\{ x\in \R^{N}: \e x\in \Lambda\}$ and $|(u, v)|:=\sqrt{u^{2}+v^{2}}$ for any $u, v\in \R$.

For any $\e>0$, we introduce the fractional space
$$
\X_{\e}:=\left\{(u, v)\in H^{s}(\R^{N})\times H^{s}(\R^{N}): \int_{\R^{N}} (V(\e x) |u|^{2}+W(\e x) |v|^{2}) dx<\infty\right\}
$$
endowed with the norm
$$
\|(u, v)\|^{2}_{\e}:=\int_{\R^{N}} |(-\Delta)^{\frac{s}{2}} u|^{2}+ |(-\Delta)^{\frac{s}{2}} v|^{2} dx+\int_{\R^{N}} (V(\e x) u^{2}+W(\e x) v^{2}) dx.
$$
Let us introduce the Euler-Lagrange functional associated with \eqref{P''}, that is
$$
\J_{\e}(u, v)=\frac{1}{2}\|(u, v)\|_{\e}^{2}-\int_{\R^{N}} H(\e x, u, v) dx
$$
for any $(u, v)\in \X_{\e}$.\\
We define 
$$
\N_{\e}:=\{(u, v)\in \X_{\e}\setminus \{(0, 0)\}:  \langle \J'_{\e}(u, v),(u, v)\rangle=0\}.
$$
It is standard to check that for any $(u, v)\in \X_{\e}\setminus \{(0, 0)\}$, the function $t\rightarrow \J_{\e}(tu, tv)$ achieves its maximum at a unique $t_{u}>0$ such that $t_{u}(u, v)\in \N_{\e}$.\\
Let us observe that $\J_{\e}\in C^{1}(\X_{\e}, \R)$ has a mountain pass geometry, that is
\begin{compactenum}[$(MP1)$]
\item $\J_{\e}(0, 0)=0$;
\item there exist $r, \rho>0$ such that $\J_{\e}(u, v)\geq r$ for $\|(u, v)\|_{\xi}=\rho$; 
\item there exists $e\in \X_{0}$ with $\|e\|_{\e}>\rho$ such that $\J_{\e}(e)< 0$.
\end{compactenum}
Indeed, using \eqref{2.4A}-\eqref{2.6A}, we can see that
$$
\int_{\R^{N}} H(\e x, u, v) dx\leq \int_{\Lambda_{\e}} H(\e x, u, v) dx+\frac{1}{k} \int_{\R^{N}\setminus \Lambda_{\e}} V(\e x) u^{2}+W(\e x) v^{2} dx
$$
which together with $(Q2)$ and Theorem \ref{Sembedding} yields
\begin{align*}
\J_{\e}(u, v)&\geq \left(\frac{1}{2}-\frac{1}{k} \right) \|(u, v)\|^{2}_{\e}-C\int_{\R^{N}} |u|^{p}+|v|^{p} dx \\
&\geq \left(\frac{1}{2}-\frac{1}{k} \right) \|(u, v)\|^{2}_{\e}-C\|(u, v)\|^{p}_{\e}
\end{align*}
where $k>2$ is fixed. Hence, $(MP2)$ holds.
On the other hand, for any $(\phi_{1}, \phi_{2})\in \X_{\e}$ such that $Q(\phi_{1}, \phi_{2})\geq 0$ and $Q(\phi_{1}, \phi_{2})\not\equiv 0$, we have, in view of $(Q1)$, that
\begin{align*}
\J_{\e}(t\phi_{1}, t\phi_{2})&\leq \frac{t^{2}}{2}\|(\phi_{1}, \phi_{2})\|^{2}_{\e}-Ct^{p} \int_{\Lambda_{\e}} Q(\phi_{1}, \phi_{2}) dx \rightarrow -\infty \mbox{ as } t\rightarrow \infty.
\end{align*}

Moreover, $\J_{\e}$ satisfies the Palais-Smale compactness condition:
\begin{lem}\label{lemma3.2-1}
Any sequence $\{(u_{n}, v_{n})\}$ in $\X_{\e}$ such that $\{\J_{\e}(u_{n}, v_{n})\}$ is bounded and $\J_{\e}'(u_{n}, v_{n})\rightarrow 0$, admits a convergent subsequence in $\X_{\e}$.
\end{lem}

\begin{proof} 
First of all, we show that $\{(u_{n}, v_{n})\}$ is bounded in $\X_{\e}$. Indeed, using conditions \eqref{2.4A}-\eqref{2.5A}, it follows that
\begin{align}\begin{split}\label{3.1-1}
\int_{\R^{N}} &|(-\Delta)^{\frac{s}{2}} u_{n}|^{2}+ |(-\Delta)^{\frac{s}{2}} v_{n}|^{2} + V(\e x)u_{n}^{2} + W(\e x)v_{n}^{2} \, dx \\
&\geq \int_{\Lambda_{\e}} u_{n} H_{u}(\e x, u_{n}, v_{n}) + \int_{\Lambda_{\e}} v_{n} H_{v}(\e x, u_{n}, v_{n}) + o(\| (u_{n}, v_{n})\|_{\e}). 
\end{split}\end{align} 
On the other hand, 
\begin{equation*}
\frac{1}{2} \int_{\R^{N}} |(-\Delta)^{\frac{s}{2}} u_{n}|^{2}+ |(-\Delta)^{\frac{s}{2}} v_{n}|^{2} + V(\e x)u_{n}^{2} + W(\e x)v_{n}^{2} \, dx= \int_{\R^{N}} H(\e x, u_{n}, v_{n}) \, dx + O(1), 
\end{equation*}
so, in view of \eqref{2.6A}, we can see that
\begin{align}\begin{split}\label{3.2-1}
\frac{1}{2} \int_{\R^{N}} &|(-\Delta)^{\frac{s}{2}} u_{n}|^{2}+ |(-\Delta)^{\frac{s}{2}} v_{n}|^{2} + V(\e x)u_{n}^{2} + W(\e x)v_{n}^{2} \, dx\\
&\leq \int_{\Lambda_{\e}} H(\e x, u_{n}, v_{n}) \, dx + \frac{1}{2k} \int_{\R^{N}\setminus \Lambda_{\e}} [V(\e x) u_{n}^{2} + W(\e x) v_{n}^{2}] + O(1). 
\end{split}\end{align}
Taking into account \eqref{2.4A}, \eqref{3.1-1} and \eqref{3.2-1} we have
\begin{align*}
\left( \frac{1}{2}- \frac{1}{p}\right) &\int_{\R^{N}} |(-\Delta)^{\frac{s}{2}} u_{n}|^{2}+ |(-\Delta)^{\frac{s}{2}} v_{n}|^{2} + V(\e x)u_{n}^{2} + W(\e x)v_{n}^{2} \, dx \\
&\leq \frac{1}{2k} \int_{\R^{N}\setminus \Lambda_{\e}} [V(\e x) u_{n}^{2} + W(\e x) v_{n}^{2}] + o(\|(u_{n}, v_{n})\|_{\e})+ O(1).
\end{align*}
Choosing $k$ such that $k>\frac{1}{2}(\frac{1}{2}- \frac{1}{p})^{-1}$ it follows that $\{(u_{n}, v_{n})\}$ is bounded. Since $\X_{\e}$ is reflexive, there exists $(u, v)\in \X_{\e}$ and a subsequence, still denoted by $\{(u_{n}, v_{n})\}$, such that $\{(u_{n}, v_{n})\}$ is weakly convergent to $(u, v)$ and $u_{n}\rightarrow u$, $v_{n}\rightarrow v$ in $L^{q}_{loc}(\R^{N})$ for any $q\in [1, 2^{*}_{s})$. Then, it is easy to check that $(u, v)$ is a critical point of $\J_{\e}$ and
\begin{align}\label{io}
\|(u, v)\|^{2}_{\e}=\int_{\R^{N}} u H_{u}(\e x, u, v)+ v H_{v}(\e x, u, v)\, dx.
\end{align}
Now we show that $\{(u_{n}, v_{n})\}$ strongly converges to $(u, v)$. To do this, we will prove the following claim. \\
\textsc{Claim 1.} For each $\delta>0$, there exists $R>0$ such that
\begin{equation}\label{3.3-1}
\limsup_{n\rightarrow \infty} \int_{\R^{N} \setminus B_{R}} |(-\Delta)^{\frac{s}{2}} u_{n}|^{2}+ |(-\Delta)^{\frac{s}{2}} v_{n}|^{2} + V(\e x)u_{n}^{2} + W(\e x)v_{n}^{2} \, dx \leq \delta, 
\end{equation}
where $B_{R}$ denotes the ball with center at $0$ and radius $R$. \\
First of all, we may assume that $R$ is chosen so that $\Lambda_{\e} \subset B_{R}$. Let $\eta_{R}$ be a cut-off function such that $\eta_{R}=0$ on $B_{R}$, $\eta_{R}=1$ on $\R^{N} \setminus B_{2R}$, $0\leq \eta\leq 1$ and $|\nabla \eta_{R}|\leq \frac{c}{R}$. Since $\{(u_{n}, v_{n})\}$ is a bounded (PS) sequence, we have
\begin{equation*}
\langle \J_{\e}'(u_{n}, v_{n}), (\eta_{R} u_{n}, \eta_{R} v_{n})  \rangle = o_{n}(1). 
\end{equation*}
Thus, using \eqref{2.6A} with $k>1$, we have 
\begin{align*}
&\iint_{\R^{2N}} \eta_{R}(x) \left[\frac{|u_{n}(x)-u_{n}(y)|^{2}}{|x-y|^{N+2s}}+\frac{|v_{n}(x)-v_{n}(y)|^{2}}{|x-y|^{N+2s}}\right]  dx dy\\
&+\iint_{\R^{2N}} \frac{(\eta_{R}(x)-\eta_{R}(y))(u_{n}(x)-u_{n}(y))}{|x-y|^{N+2s}} u_{n}(y) dx dy \\
&+ \int_{\R^{N}} (V(\e x)u_{n}^{2} + W(\e x) v_{n}^{2})\eta_{R} \, dx \\
&=\int_{\R^{N}} (u_{n}H_{u}(\e x, u_{n}, v_{n}) +v_{n}H_{v}(\e x, u_{n}, v_{n}))\eta_{R} + o_{n}(1)\\
&\leq \frac{1}{k} \int_{\R^{N}} V(\e x)u_{n}^{2} + W(\e x) v_{n}^{2} \, dx + o_{n}(1), 
\end{align*}
from which we deduce that
\begin{align}\label{jt}
&\left(1-\frac{1}{k}\right)\int_{\R^{N}\setminus B_{2R}} |(-\Delta)^{\frac{s}{2}}u_{n}|^{2} + |(-\Delta)^{\frac{s}{2}} v_{n}|^{2}+ V(\e x)u_{n}^{2} + W(\e x) v_{n}^{2} \, dx \nonumber \\
&\leq - \iint_{\R^{2N}} \frac{(\eta_{R}(x)-\eta_{R}(y))(u_{n}(x)-u_{n}(y))}{|x-y|^{N+2s}} u_{n}(y) dx dy+ o_{n}(1).
\end{align}
Using the H\"older inequality and the boundedness of $\{(u_{n}, v_{n})\}$ in $\X_{\e}$, we get
\begin{align*}
&\left| \iint_{\R^{2N}} \frac{(\eta_{R}(x)-\eta_{R}(y))(u_{n}(x)-u_{n}(y))}{|x-y|^{N+2s}} u_{n}(y) dx dy\right| \\
&\leq \left( \iint_{\R^{2N}}\frac{|\eta_{R}(x)-\eta_{R}(y)|^{2}}{|x-y|^{N+2s}} u^{2}_{n}(y) dx dy\right)^{\frac{1}{2}} \left(\iint_{\R^{2N}} \frac{|u_{n}(x)-u_{n}(y)|^{2}}{|x-y|^{N+2s}} dx dy\right)^{\frac{1}{2}} \\
&\leq C  \left( \iint_{\R^{2N}}\frac{|\eta_{R}(x)-\eta_{R}(y)|^{2}}{|x-y|^{N+2s}} u^{2}_{n}(y) dx dy\right)^{\frac{1}{2}}.
\end{align*} 
Therefore, if we prove that 
\begin{align}\label{JT}
\lim_{R\rightarrow \infty} \limsup_{n\rightarrow \infty} \iint_{\R^{2N}}\frac{|\eta_{R}(x)-\eta_{R}(y)|^{2}}{|x-y|^{N+2s}} u^{2}_{n}(x) dx dy=0,
\end{align}
we can use \eqref{jt} to conclude that Claim $1$ holds true.\\
Then, in what follows, we show that \eqref{JT} is satisfied. Firstly, we note that $\R^{2N}$ can be written as 
$$
\R^{2N}=((\R^{N}\setminus B_{2R})\times (\R^{N}\setminus B_{2R})) \cup ((\R^{N}\setminus B_{2R})\times B_{2R})\cup (B_{2R}\times \R^{N})=: X^{1}_{R}\cup X^{2}_{R} \cup X^{3}_{R}.
$$
Hence,
\begin{align}\label{Pa1}
&\iint_{\R^{2N}}\frac{|\eta_{R}(x)-\eta_{R}(y)|^{2}}{|x-y|^{N+2s}} u^{2}_{n}(x) dx dy =\iint_{X^{1}_{R}}\frac{|\eta_{R}(x)-\eta_{R}(y)|^{2}}{|x-y|^{N+2s}} u^{2}_{n}(x) dx dy \nonumber \\
&+\iint_{X^{2}_{R}}\frac{|\eta_{R}(x)-\eta_{R}(y)|^{2}}{|x-y|^{N+2s}} u^{2}_{n}(x) dx dy+
\iint_{X^{3}_{R}}\frac{|\eta_{R}(x)-\eta_{R}(y)|^{2}}{|x-y|^{N+2s}} u^{2}_{n}(x) dx dy.
\end{align}
Now, we estimate each integral in (\ref{Pa1}).
Since $\eta_{R}=1$ in $\R^{N}\setminus B_{2R}$, we get
\begin{align}\label{Pa2}
\iint_{X^{1}_{R}}\frac{|u_{n}(x)|^{2}|\eta_{R}(x)-\eta_{R}(y)|^{2}}{|x-y|^{N+2s}} dx dy=0.
\end{align}
Let $K>4$. Clearly, we have
\begin{equation*}
X^{2}_{R}=(\R^{N} \setminus B_{2R})\times B_{2R} \subset ((\R^{N}\setminus B_{KR})\times B_{2R})\cup ((B_{KR}\setminus B_{2R})\times B_{2R}).
\end{equation*}
Let us observe that, if $(x, y) \in (\R^{N}\setminus B_{kR})\times B_{2R}$, then
\begin{equation*}
|x-y|\geq |x|-|y|\geq |x|-2R>\frac{|x|}{2}. 
\end{equation*}
Then, taking into account $0\leq \eta_{R}\leq 1$, $|\nabla \eta_{R}|\leq \frac{C}{R}$ and applying the H\"older inequality, we have
\begin{align}\label{Pa3}
&\iint_{X^{2}_{R}}\frac{|u_{n}(x)|^{2}|\eta_{R}(x)-\eta_{R}(y)|^{2}}{|x-y|^{N+2s}} dx dy \nonumber \\
&=\int_{\R^{N}\setminus B_{KR}} \int_{B_{2R}} \frac{|u_{n}(x)|^{2}|\eta_{R}(x)-\eta_{R}(y)|^{2}}{|x-y|^{N+2s}} dx dy + \int_{B_{KR}\setminus B_{2R}} \int_{B_{2R}} \frac{|u_{n}(x)|^{2}|\eta_{R}(x)-\eta_{R}(y)|^{2}}{|x-y|^{N+2s}} dx dy \nonumber \\
&\leq 2^{2+N+2s} \int_{\R^{N}\setminus B_{KR}} \int_{B_{2R}} \frac{|u_{n}(x)|^{2}}{|x|^{N+2s}}\, dxdy+ \frac{C}{R^{2}} \int_{B_{KR}\setminus B_{2R}} \int_{B_{2R}} \frac{|u_{n}(x)|^{2}}{|x-y|^{N+2(s-1)}}\, dxdy \nonumber \\
&\leq CR^{N} \int_{\R^{N}\setminus B_{KR}} \frac{|u_{n}(x)|^{2}}{|x|^{N+2s}}\, dx + \frac{C}{R^{2}} (KR)^{2(1-s)} \int_{B_{KR}\setminus B_{2R}} |u_{n}(x)|^{2} dx \nonumber \\
&\leq CR^{N} \left( \int_{\R^{N}\setminus B_{KR}} |u_{n}(x)|^{2^{*}_{s}} dx \right)^{\frac{2}{2^{*}_{s}}} \left(\int_{\R^{N}\setminus B_{KR}}\frac{1}{|x|^{\frac{N^{2}}{2s} +N}}\, dx \right)^{\frac{2s}{N}} + \frac{C K^{2(1-s)}}{R^{2s}} \int_{B_{KR}\setminus B_{2R}} |u_{n}(x)|^{2} dx \nonumber \\
&\leq \frac{C}{K^{N}} \left( \int_{\R^{N}\setminus B_{KR}} |u_{n}(x)|^{2^{*}_{s}} dx \right)^{\frac{2}{2^{*}_{s}}} + \frac{C K^{2(1-s)}}{R^{2s}} \int_{B_{KR}\setminus B_{2R}} |u_{n}(x)|^{2} dx \nonumber \\
&\leq \frac{C}{K^{N}}+ \frac{C K^{2(1-s)}}{R^{2s}} \int_{B_{KR}\setminus B_{2R}} |u_{n}(x)|^{2} dx.
\end{align}

\noindent
Now, take $\e\in (0,1)$, and we observe that
\begin{align}\label{Ter1}
&\iint_{X^{3}_{R}} \frac{|u_{n}(x)|^{2} |\eta_{R}(x)- \eta_{R}(y)|^{2}}{|x-y|^{N+2s}}\, dxdy \nonumber\\
&\leq \int_{B_{2R}\setminus B_{\varepsilon R}} \int_{\R^{N}} \frac{|u_{n}(x)|^{2} |\eta_{R}(x)- \eta_{R}(y)|^{2}}{|x-y|^{N+2s}}\, dxdy + \int_{B_{\varepsilon R}} \int_{\R^{N}} \frac{|u_{n}(x)|^{2} |\eta_{R}(x)- \eta_{R}(y)|^{2}}{|x-y|^{N+2s}}\, dxdy. 
\end{align}
Let us estimate the integrals on the right hand side in \eqref{Ter1}. Then we can see that
\begin{align*}
\int_{B_{2R}\setminus B_{\varepsilon R}} \int_{\R^{N} \cap \{y: |x-y|<R\}} \frac{|u_{n}(x)|^{2} |\eta_{R}(x)- \eta_{R}(y)|^{2}}{|x-y|^{N+2s}}\, dxdy \leq \frac{C}{R^{2s}} \int_{B_{2R}\setminus B_{\varepsilon R}} |u_{n}(x)|^{2} dx
\end{align*}
and 
\begin{align*}
\int_{B_{2R}\setminus B_{\varepsilon R}} \int_{\R^{N} \cap \{y: |x-y|\geq R\}} \frac{|u_{n}(x)|^{2} |\eta_{R}(x)- \eta_{R}(y)|^{2}}{|x-y|^{N+2s}}\, dxdy \leq \frac{C}{R^{2s}} \int_{B_{2R}\setminus B_{\varepsilon R}} |u_{n}(x)|^{2} dx
\end{align*}
which yield
\begin{align}\label{Ter2}
\int_{B_{2R}\setminus B_{\varepsilon R}} \int_{\R^{N}} \frac{|u_{n}(x)|^{2} |\eta_{R}(x)- \eta_{R}(y)|^{2}}{|x-y|^{N+2s}}\, dxdy \leq \frac{C}{R^{2s}} \int_{B_{2R}\setminus B_{\varepsilon R}} |u_{n}(x)|^{2} dx. 
\end{align}
Now, using the definition of $\eta_{R}$, $\e\in (0,1)$, and $0\leq \eta_{R}\leq 1$, we obtain 
\begin{align}\label{Ter3}
\int_{B_{\varepsilon R}} \int_{\R^{N}} \frac{|u_{n}(x)|^{2} |\eta_{R}(x)- \eta_{R}(y)|^{2}}{|x-y|^{N+2s}}\, dxdy &= \int_{B_{\varepsilon R}} \int_{\R^{N}\setminus B_{R}} \frac{|u_{n}(x)|^{2} |\eta_{R}(x)- \eta_{R}(y)|^{2}}{|x-y|^{N+2s}}\, dxdy\nonumber \\
&\leq 4 \int_{B_{\varepsilon R}} \int_{\R^{N}\setminus B_{R}} \frac{|u_{n}(x)|^{2}}{|x-y|^{N+2s}}\, dxdy\nonumber \\
&\leq C \int_{B_{\varepsilon R}} |u_{n}(x)|^{2} dx \int_{(1-\e)R}^{\infty} \frac{1}{r^{1+2s}} dr\nonumber \\
&=\frac{C}{[(1-\e)R]^{2s}} \int_{B_{\varepsilon R}} |u_{n}(x)|^{2} dx
\end{align}
where we used the fact that if $(x, y) \in B_{\varepsilon R}\times (\R^{N} \setminus B_{R})$, then $|x-y|>(1-\e)R$. \\
Taking into account \eqref{Ter1}, \eqref{Ter2} and \eqref{Ter3} we deduce 
\begin{align}\label{Pa4}
\iint_{X^{3}_{R}} &\frac{|u_{n}(x)|^{2} |\eta_{R}(x)- \eta_{R}(y)|^{2}}{|x-y|^{N+2s}}\, dxdy \nonumber \\
&\leq \frac{C}{R^{2s}} \int_{B_{2R}\setminus B_{\varepsilon R}} |u_{n}(x)|^{2} dx + \frac{C}{[(1-\e)R]^{2s}} \int_{B_{\varepsilon R}} |u_{n}(x)|^{2} dx. 
\end{align}
Putting together \eqref{Pa1}, \eqref{Pa2}, \eqref{Pa3} and \eqref{Pa4}, we can infer 
\begin{align}\label{Pa5}
\iint_{\R^{2N}} &\frac{|u_{n}(x)|^{2} |\eta_{R}(x)- \eta_{R}(y)|^{2}}{|x-y|^{N+2s}}\, dxdy \nonumber \\
&\leq \frac{C}{K^{N}} + \frac{CK^{2(1-s)}}{R^{2s}} \int_{B_{KR}\setminus B_{2R}} |u_{n}(x)|^{2} dx + \frac{C}{R^{2s}} \int_{B_{2R}\setminus B_{\varepsilon R}} |u_{n}(x)|^{2} dx + \frac{C}{[(1-\e)R]^{2s}}\int_{B_{\varepsilon R}} |u_{n}(x)|^{2} dx. 
\end{align}
Since $\{u_{n}\}$ is bounded in $H^{s}(\R^{N})$, by Theorem \ref{Sembedding} we may assume that $u_{n}\rightarrow u$ in $L^{2}_{loc}(\R^{N})$ for some $u\in H^{s}(\R^{N})$. Then, taking the limit as $n\rightarrow \infty$ in \eqref{Pa5}, we have
\begin{align*}
&\limsup_{n\rightarrow \infty} \iint_{\R^{2N}} \frac{|u_{n}(x)|^{2} |\eta_{R}(x)- \eta_{R}(y)|^{2}}{|x-y|^{N+2s}}\, dxdy\\
&\leq \frac{C}{K^{N}} + \frac{CK^{2(1-s)}}{R^{2s}} \int_{B_{KR}\setminus B_{2R}} |u(x)|^{2} dx + \frac{C}{R^{2s}} \int_{B_{2R}\setminus B_{\varepsilon R}} |u(x)|^{2} dx + \frac{C}{[(1-\e)R]^{2s}}\int_{B_{\varepsilon R}} |u(x)|^{2} dx \\
&\leq \frac{C}{K^{N}} + CK^{2} \left( \int_{B_{KR}\setminus B_{2R}} |u(x)|^{2^{*}_{s}} dx\right)^{\frac{2}{2^{*}_{s}}} + C\left(\int_{B_{2R}\setminus B_{\varepsilon R}} |u(x)|^{2^{*}_{s}} dx\right)^{\frac{2}{2^{*}_{s}}} + C\left( \frac{\e}{1-\e}\right)^{2s} \left(\int_{B_{\varepsilon R}} |u(x)|^{2^{*}_{s}} dx\right)^{\frac{2}{2^{*}_{s}}}, 
\end{align*}
where in the last passage we used the H\"older inequality. \\
Since $u\in L^{2^{*}_{s}}(\R^{N})$, $K>4$ and $\e \in (0,1)$, we obtain
\begin{align*}
\limsup_{R\rightarrow \infty} \int_{B_{KR}\setminus B_{2R}} |u(x)|^{2^{*}_{s}} dx = \limsup_{R\rightarrow \infty} \int_{B_{2R}\setminus B_{\varepsilon R}} |u(x)|^{2^{*}_{s}} dx = 0. 
\end{align*}
Choosing $\e= \frac{1}{K}$, we get
\begin{align*}
&\limsup_{R\rightarrow \infty} \limsup_{n\rightarrow \infty} \iint_{\R^{2N}} \frac{|u_{n}(x)|^{2} |\eta_{R}(x)- \eta_{R}(y)|^{2}}{|x-y|^{N+2s}}\, dxdy\\
&\leq \lim_{K\rightarrow \infty} \limsup_{R\rightarrow \infty} \Bigl[\, \frac{C}{K^{N}} + CK^{2} \left( \int_{B_{KR}\setminus B_{2R}} |u(x)|^{2^{*}_{s}} dx\right)^{\frac{2}{2^{*}_{s}}} + C\left(\int_{B_{2R}\setminus B_{\frac{1}{K} R}} |u(x)|^{2^{*}_{s}} dx\right)^{\frac{2}{2^{*}_{s}}} \\
&\quad+ C\left(\frac{1}{K-1}\right)^{2s} \left(\int_{B_{\frac{1}{K} R}} |u(x)|^{2^{*}_{s}} dx\right)^{\frac{2}{2^{*}_{s}}}\, \Bigr]\\
&\leq \lim_{k\rightarrow \infty} \frac{C}{K^{N}} + C\left(\frac{1}{K-1}\right)^{2s} \left(\int_{\R^{N}} |u(x)|^{2^{*}_{s}} dx \right)^{\frac{2}{2^{*}_{s}}}= 0
\end{align*}
that is \eqref{JT} holds true.  \\
Then, using \eqref{3.3-1} and \eqref{2.6A} of Lemma \ref{lem2.2Alves}, we obtain that 
\begin{equation}\label{3.6-1}
\int_{\R^{N} \setminus B_{R}} u_{n} H_{u}(\e x, u_{n},v_{n})+ v_{n}H_{v}(\e x, u_{n}, v_{n})\, dx \leq \frac{\delta}{4}, 
\end{equation}
for any $n$ big enough.
On the other hand, taking $R$ larger if necessary, we can suppose that
\begin{equation}\label{3.6-2}
\int_{\R^{N} \setminus B_{R}} u H_{u}(\e x, u, v)+ v H_{v}(\e x, u, v)\, dx \leq \frac{\delta}{4}. 
\end{equation}
Therefore, we see that \eqref{3.6-1} and \eqref{3.6-2} yield
\begin{align}\label{ioo}
\left|\int_{\R^{N} \setminus B_{R}} u_{n} H_{u}(\e x, u_{n},v_{n})+ v_{n}H_{v}(\e x, u_{n}, v_{n})\, dx- \int_{\R^{N} \setminus B_{R}} u H_{u}(\e x, u, v)+ vH_{v}(\e x, u, v)\, dx\right|\leq \frac{\delta}{2}
\end{align}
for $n$ large enough.
Now, observing that $B_{R}$ is bounded, we can use the dominated convergence theorem and the strong convergence in $L^{q}_{loc}(\R^{N})$ to deduce that 
\begin{align}\label{iooo}
\int_{B_{R}} u_{n} H_{u}(\e x, u_{n},v_{n})+ v_{n}H_{v}(\e x, u_{n}, v_{n})\, dx\rightarrow \int_{B_{R}} u H_{u}(\e x, u, v)+ v H_{v}(\e x, u, v)\, dx
\end{align}
as $n\rightarrow \infty$.
Putting together $\langle \J'_{\e}(u_{n}, v_{n}), (u_{n}, v_{n})\rangle=o_{n}(1)$, \eqref{io}, \eqref{ioo} and \eqref{iooo},  we can infer that
\begin{equation*}
\lim_{n\rightarrow \infty} \|(u_{n}, v_{n})\|_{\e}^{2}= \|(u, v)\|_{\e}^{2}
\end{equation*}
which implies that $\{(u_{n}, v_{n})\}$ strongly converges to $(u, v)$ in $\X_{\e}$.

\end{proof}

\noindent
In light of mountain pass theorem \cite{ambrosetti-rab}, there exists $(u, v)\in \X_{\e}\setminus \{0\}$ such that 
$$
\J_{\e}(u, v)=c_{\e} \mbox{ and } \J_{\e}'(u, v)=0,
$$ 
where
$$
c_{\e}=\inf_{\gamma\in \Gamma_{\e}} \max_{t\in [0, 1]} \J_{\e}(\gamma(t))>0
$$
and
$$
\Gamma_{\e}=\left\{\gamma\in C([0, 1], \X_{\e}): \gamma(0)=0, \J_{\e}(\gamma(1))\leq 0\right\}.
$$
Finally, we prove the following result.
\begin{lem}
If $(u, v)$ is a critical point of $\J_{\e}$, we have that $u, v\geq 0$ in $\R^{N}$.
\end{lem}
\begin{proof}
Since $(u, v)$ is a critical point of $\J_{\e}$, we know that for any $(\phi, \psi)\in \X_{\e}\times \X_{\e}$ it holds
\begin{align}\label{WFP}
\iint_{\R^{2N}} \frac{(u(x)-u(y))(\phi(x)-\phi(y))}{|x-y|^{N+2s}}&+\frac{(v(x)-v(y))(\psi(x)-\psi(y))}{|x-y|^{N+2s}}\, dx dy+\int_{\R^{N}}V(\e x) u \phi+W(\e x) v \psi \, dx  \nonumber\\
&=\int_{\R^{N}} \phi H_{u}(\e x, u, v)+\psi H_{v}(\e x, u, v) \, dx.
\end{align}
Taking $\phi=u^{-}$ and $\psi=v^{-}$ in \eqref{WFP}, where $x^{-}=\min\{x, 0\}$, and recalling that $(x-y)(x^{-}-y^{-})\geq |x^{-}-y^{-}|^{2}$ for any $x, y\in \R$, we can see that 
\begin{align}\label{SqC}
\iint_{\R^{2N}} \frac{|u^{-}(x)-u^{-}(y)|^{2}}{|x-y|^{N+2s}}&+\frac{|v^{-}(x)-v^{-}(y)|^{2}}{|x-y|^{N+2s}}\, dx dy+\int_{\R^{N}}V(\e x) (u^{-})^{2}+W(\e x) (v^{-})^{2} \, dx \nonumber \\
&\leq \int_{\R^{N}} u^{-} H_{u}(\e x, u, v)+v^{-} H_{v}(\e x, u, v) \, dx.
\end{align}
Now, we can note that for any $x\in \R^{N}\setminus \Lambda$,
\begin{align*}
H_{u}(x, u, v)&=\frac{\chi'(|(u, v)|)uQ(u, v)}{\sqrt{u^{2}+v^{2}}}+\chi(|(u, v)|)Q_{u}(u, v)-\frac{\chi'(|(u, v)|) Au(u^{2}+v^{2})}{\sqrt{u^{2}+v^{2}}} \\
&+(1-\chi(|(u, v)|))2uA
\end{align*}
and
\begin{align*}
H_{v}(x, u, v)&=\frac{\chi'(|(u, v)|)vQ(u, v)}{\sqrt{u^{2}+v^{2}}}+\chi(|(u, v)|)Q_{v}(u, v)-\frac{\chi'(|(u, v)|) Av(u^{2}+v^{2})}{\sqrt{u^{2}+v^{2}}}\\
&+(1-\chi(|(u, v)|))2vA.
\end{align*}
Then, we deduce that
\begin{align}\label{Aumeno}
&u^{-}H_{u}(x, u, v)+v^{-}H_{v}(x, u, v) \nonumber \\
&=\frac{\chi'(|(u, v)|)(u^{-})^{2}Q(u, v)}{\sqrt{u^{2}+v^{2}}}-\frac{\chi'(|(u, v)|) A(u^{-})^{2}(u^{2}+v^{2})}{\sqrt{u^{2}+v^{2}}} +(1-\chi(|(u, v)|))2(u^{-})^{2}A \nonumber \\
&\quad+\frac{\chi'(|(u, v)|) (v^{-})^{2}Q(u, v)}{\sqrt{u^{2}+v^{2}}} -\frac{\chi'(|(u, v)|) A(v^{-})^{2}(u^{2}+v^{2})}{\sqrt{u^{2}+v^{2}}}+(1-\chi(|(u, v)|))2(v^{-})^{2}A    \nonumber \\
&=\frac{\chi'(|(u, v)|)[(u^{-})^{2}+(v^{-})^{2}]Q(u, v)}{\sqrt{u^{2}+v^{2}}}-\frac{\chi'(|(u, v)|) A[(u^{-})^{2}+(v^{-})^{2}](u^{2}+v^{2})}{\sqrt{u^{2}+v^{2}}}  \nonumber \\
&\quad+(1-\chi(|(u, v)|))2[(u^{-})^{2}+(v^{-})^{2}]A.
\end{align}
Taking into account \eqref{Aumeno} and the definitions of $A$ and $\chi$, we can find a constant $C>0$ such that
\begin{align}\label{Aumenovmeno}
\left|u^{-}H_{u}(x, u, v)+v^{-}H_{v}(x, u, v)\right|\leq  CA[(u^{-})^{2}+(v^{-})^{2}] \mbox{ for any } x\in \R^{N}\setminus \Lambda.
\end{align}
Recalling that $A\rightarrow 0$ as $a\rightarrow 0$, we can see that \eqref{SqC} and \eqref{Aumenovmeno} imply
\begin{align}\label{SqC1}
\iint_{\R^{2N}} \frac{|u^{-}(x)-u^{-}(y)|^{2}}{|x-y|^{N+2s}}&+\frac{|v^{-}(x)-v^{-}(y)|^{2}}{|x-y|^{N+2s}} \, dx dy+\int_{\R^{N}}V(\e x) (u^{-})^{2}+W(\e x) (v^{-})^{2} \, dx \nonumber \\
&\leq C\int_{\Lambda_{\e}} u^{-} Q_{u}(u, v)+v^{-} Q_{v}(u, v) \, dx
\end{align}
for any $a$ sufficiently small. By the definition of $Q$, we know that $Q_{u}(u^{-}, v)=0=Q_{v}(u, v^{-})$, so we get
$$
\int_{\Lambda_{\e}} u^{-} Q_{u}(u, v)+v^{-} Q_{v}(u, v) \, dx= \int_{\Lambda_{\e}} u^{-} Q_{u}(u^{-}, v)+v^{-} Q_{v}(u, v^{-}) \, dx=0.
$$
Using \eqref{SqC1}, we can infer that $\|(u^{-}, v^{-})\|_{\e}^{2}=0$, that is $u^{-}=v^{-}=0$ in $\R^{N}$.

\end{proof}

\section{compactness properties}
This section is devoted to prove compactness properties related to the functional $\J_{\e}$.
Since we are interested in obtaining multiple critical points, we work with the functional $\J_{\e}$ restricted to the Nehari manifold $\N_{\e}$. 
We begin by proving some useful properties of $\N_{\e}$.
\begin{lem}\label{lemma2.2}
There exist positive constants $a_{1}$, $\delta$ such that, for each $a\in (0, a_{1})$, $(u,v)\in \N_{\e}$, there hold
\begin{equation}\label{5}
\int_{\Lambda_{\e}} Q(u, v) \, dx\geq \delta
\end{equation}
and 
\begin{equation}\label{6}
\int_{\R^{N}\setminus \Lambda_{\e}} (V(\e x) u^{2} + W(\e x) v^{2}) \, dx\leq 2p \int_{\Lambda_{\e}} Q(u, v) \, dx. 
\end{equation}
\end{lem}

\begin{proof}
Using \eqref{2.6A}, $(Q2)$ and Theorem \ref{Sembedding}, we can see that for any $(u, v)\in \N_{\e}$ it holds
\begin{align*}
\|(u, v)\|^{2}_{\e}&=\int_{\Lambda_{\e}} (u Q_{u} + v Q_{v}) \, dx+ \int_{\R^{N}\setminus \Lambda_{\e}} (u H_{u} + vH_{v}) \, dx \\
&\leq C \int_{\Lambda_{\e}} (|u|^{p}+|v|^{p}) \, dx+\frac{1}{2} \int_{\R^{N}} (V(\e x) u^{2} + W(\e x) v^{2}) \, dx \\
&\leq C \|(u, v)\|^{p}_{\e}+\frac{1}{2} \|(u, v)\|^{2}_{\e}
\end{align*} 
which implies that there is $\hat{\delta}>0$ such that
\begin{equation*}
\|(u, v)\|_{\e} \geq \hat{\delta} \, \mbox{ for any } (u, v)\in \N_{\e}.
\end{equation*}
Thus, using \eqref{2.1} and \eqref{2.6A} (with $k=2$), we obtain
\begin{align*}
\hat{\delta}^{2} \leq \|(u,v)\|_{\e}^{2} &= \int_{\Lambda_{\e}} (u Q_{u} + v Q_{v}) \, dx+ \int_{\R^{N}\setminus \Lambda_{\e}} (u H_{u} + vH_{v}) \, dx\\
&\leq p \int_{\Lambda_{\e}} Q(u, v) \, dx+ \frac{1}{2} \int_{\R^{N}\setminus \Lambda_{\e}} (V(\e x) u^{2} + W(\e x) v^{2}) \, dx,  
\end{align*}
which gives
\begin{equation*}
\frac{\hat{\delta}^{2}}{2}\leq \frac{1}{2} \|(u,v)\|_{\e}^{2} \leq p \int_{\Lambda_{\e}} Q(u, v) \, dx. 
\end{equation*}
Therefore, \eqref{5} holds with $\delta= \frac{\hat{\delta}^{2}}{2p}$. \\
Now, taking into account $(u,v)\in \N_{\e}$, \eqref{2.1} and \eqref{2.6A}, we get
\begin{align*}
\int_{\R^{N}\setminus \Lambda_{\e}} (V(\e x) u^{2} + W(\e x) v^{2}) \, dx &\leq \int_{\R^{N}\setminus \Lambda_{\e}} (uH_{u} + vH_{v}) \, dx+ \int_{\Lambda_{\e}} (uQ_{u} + vQ_{v}) \,dx\\
&\leq \frac{1}{2} \int_{\R^{N}\setminus \Lambda_{\e}} (V(\e x)u^{2} + W(\e x) v^{2} ) \, dx+ p\int_{\Lambda_{\e}} Q(u,v) \, dx, 
\end{align*}
which implies that \eqref{6} is satisfied.  
\end{proof}

\noindent
Now we aim to show that the functional $\J_{\e}$ restricted to $\N_{\e}$, satisfies the Palais-Smale condition. To achieve our goal we prove the following technical lemma. 

\begin{lem}\label{lemma2.3}
Let $\phi_{\e}: \mathbb{H}_{\e} \rightarrow \R$ be given by
\begin{equation*}
\phi_{\e}(u, v):= \|(u,v)\|_{\e}^{2} - \int_{\R^{N}} (uH_{u}(\e x, u, v) + vH_{v}(\e x, u,v)) \, dx. 
\end{equation*}
Then, there exist $a_{2}, b>0$ such that, for each $a\in (0, a_{2})$, 
\begin{equation}\label{7}
\langle\phi_{\e}'(u,v),(u,v)\rangle\leq -b<0 \, \mbox{ for each } (u, v)\in \N_{\e}. 
\end{equation}
\end{lem}

\begin{proof}
Given $(u,v)\in \N_{\e}$, we can use the definition of $H$, \eqref{2.1} and \eqref{2.11} to get
\begin{align}\begin{split}\label{8}
\langle\phi_{\e}'(u,v),(u,v)\rangle &= \int_{\Lambda_{\e}} (uQ_{u} + vQ_{v})-(u^{2}Q_{uu} + v^{2}Q_{vv} +2uvQ_{uv}) \, dx\\
&+\int_{\R^{N}\setminus \Lambda_{\e}} (uH_{u}+ vH_{v}) \, dx- \int_{\R^{N}\setminus \Lambda_{\e}} (u^{2}H_{uu} + v^{2}H_{vv} +2uvH_{uv}) \, dx\\
&=-p(p-2) \int_{\Lambda_{\e}} Q(u,v)\, dx + \int_{\R^{N}\setminus \Lambda_{\e}} D_{1} \, dx- \int_{\R^{N}\setminus \Lambda_{\e}} D_{2} \, dx
\end{split}\end{align}
where 
\begin{equation*}
D_{1}:= (uH_{u}+ vH_{v}) \, \mbox{ and } \, D_{2}:= (u^{2}Q_{uu} + v^{2}Q_{vv} +2uvQ_{uv}). 
\end{equation*}
We set $|z|=\sqrt{u^{2}+v^{2}}$. By the definitions of $\hat{Q}$ and $\eta$, and using \eqref{2.1} again, we can see that
\begin{align*}
|D_{1}|&= \left| \eta' \frac{Q}{|z|} + p\eta \frac{Q}{|z|^{2}} - A\eta'|z|+ 2A(1-\eta)\right| |z|^{2}\\
&\leq \left( \frac{C}{a} 5aA + pA + A\frac{C}{a} 5a + 4A\right) |z|^{2}\\
&\leq C_{1}A|z|^{2}. 
\end{align*}
Since $A\rightarrow 0$ as $a\rightarrow 0^{+}$, the last inequality together with $(H_{3})$ implies that
\begin{equation}\label{9}
\int_{\R^{N}\setminus \Lambda_{\e}} (uH_{u}+vH_{v}) \, dx\leq o(1) \int_{\R^{N}\setminus \Lambda_{\e}} (V(\e x) u^{2} + W(\e x) v^{2}) \, dx,
\end{equation}
where $o(1)\rightarrow 0$ as $a\rightarrow 0^{+}$. \\
Now we aim to estimate the last integral in \eqref{8}. Firstly we observe that 
\begin{equation}\label{10}
D_{2}= -A\eta'(|z|^{2}+ 4|z|)|z|^{2} + 2A(1-\eta) |z|^{2} + \eta''Q|z| |z|^{2} + D_{3}+ D_{4},
\end{equation}
with 
\begin{equation*}
D_{3}:= \frac{2\eta'}{|z|} (u^{3}Q_{u} + v^{3}Q_{v} + u^{2}vQ_{v} + uv^{2}Q_{u})
\end{equation*}
and 
\begin{equation*}
D_{4}:= \eta (u^{2}Q_{uu} + v^{2}Q_{vv} +2uvQ_{uv}). 
\end{equation*}
Thanks to \eqref{3}, we obtain that
\begin{equation*}
|A\eta'(|z|^{2} +4|z|) |z|^{2}|\leq A\frac{C}{a} (25a^{2} + 20 a)|z|^{2}= o(1) |z|^{2}. 
\end{equation*}
On the other hand, by the definition of $A$, we have 
\begin{equation*}
2A(1-\eta)|z|^{2}=o(1) |z|^{2} \, \mbox{ and } \, \eta''Q|z||z|^{2} = o(1) |z|^{2}. 
\end{equation*}
It follows from \eqref{2.1} that 
\begin{equation*}
|D_{3}| = |4p\eta'Q||z|\leq 4p\frac{C}{a}A|z|^{2} 5a = 20pCA|z|^{2}=o(1)|z|^{2}, 
\end{equation*}
and \eqref{2.11} implies that
\begin{equation*}
D_{4}= \eta (u^{2} Q_{uu} + v^{2}Q_{vv} + 2uvQ_{uv})= \eta p(p-1)Q\geq 0. 
\end{equation*}
Taking into account the above estimates, we deduce that
\begin{equation*}
\int_{\R^{N}\setminus \Lambda_{\e}} (u^{2}H_{uu} + v^{2}H_{vv} + 2uvH_{uv}) \, dx\leq o(1) \int_{\R^{N}\setminus \Lambda_{\e}} (V(\e x)u^{2} + W(\e x)v^{2}) \, dx. 
\end{equation*}
Thus, by \eqref{8} and \eqref{9}, we get 
\begin{equation*}
\langle\phi_{\e}'(u,v),(u,v)\rangle\leq -p(p-2) \int_{\Lambda_{\e}} Q(u,v)\, dx + o(1) \int_{\R^{N}\setminus \Lambda_{\e}} (V(\e x)u^{2} + W(\e x)v^{2}) \, dx. 
\end{equation*} 
Applying Lemma \ref{lemma2.2} we have, for $a$ small enough, 
\begin{equation*}
\langle\phi_{\e}'(u,v),(u,v)\rangle\leq (-p(p-2) +o(1))\int_{\Lambda_{\e}} Q(u,v) \, dx\leq -\frac{p(p-2)}{2}\delta = -b<0. 
\end{equation*}
 
\end{proof}

\noindent
At this point, we are able to deduce the following compactness result.
\begin{prop}\label{prop1}
The functional $\J_{\e}$ restricted to $\N_{\e}$ satisfies $(PS)_{c}$ for each $c\in \R$. 
\end{prop}

\begin{proof}
Let $\{(u_{n}, v_{n})\}\subset \N_{\e}$ be such that
\begin{equation*}
\J_{\e}(u_{n}, v_{n})\rightarrow c \, \mbox{ and } \, \|\J_{\e}'(u_{n}, v_{n})\|_{*}=o_{n}(1), 
\end{equation*}
where $o_{n}(1)$ goes to zero when $n\rightarrow \infty$. Then, there exists $\{\lambda_{n}\}\subset \R$ satisfying
\begin{equation*}
\J_{\e}'(u_{n}, v_{n})= \lambda_{n} \phi_{\e}'(u_{n}, v_{n})+ o_{n}(1), 
\end{equation*}
with $\phi_{\e}$ as in Lemma \ref{lemma2.3}. Due to the fact that $(u_{n}, v_{n})\in \N_{\e}$, we get
\begin{equation}\label{Ter}
0= \langle \J_{\e}'(u_{n}, v_{n}), (u_{n}, v_{n})\rangle= \lambda_{n} \langle\phi_{\e}'(u_{n}, v_{n}),(u_{n}, v_{n})\rangle + o_{n}(1) \|(u_{n}, v_{n})\|_{\e}. 
\end{equation}
Proceeding as in the proof of Lemma \ref{lemma3.2-1} we can see that there exists $C>0$ such that 
$$
\|(u_{n}, v_{n})\|_{\e}\leq C \mbox{ for any } n\in \mathbb{N}.
$$ 
On the other hand, from Lemma \ref{lemma2.3}, we may assume that $\langle\phi_{\e}'(u_{n}, v_{n}),(u_{n}, v_{n})\rangle\rightarrow \ell<0$. \\
Then, in view of \eqref{Ter}, we can deduce that $\lambda_{n}\rightarrow 0$ and that $\J_{\e}'(u_{n}, v_{n})\rightarrow 0$ in the dual space of $\X_{\e}$. Invoking Lemma \ref{lemma3.2-1} we can infer that $\{(u_{n}, v_{n})\}$ admits a convergent subsequence in $\X_{\e}$. 
\end{proof}


\section{barycenter map and multiplicity of solutions to \eqref{P''}}

In this section our main purpose is to apply the Ljusternik-Schnirelmann category theory to prove a multiplicity result for system \eqref{P''}. In order to accomplish our goal, we first give some useful lemmas.
We start by proving the following result.
\begin{lem}\label{lem3.1}
Let $\e_{n}\rightarrow 0^{+}$ and $\{(u_{n}, v_{n})\}\subset \N_{\e_{n}}$ be such that $\J_{\e_{n}}(u_{n}, v_{n})\rightarrow C^{*}$. Then there exists $\{\tilde{y}_{n}\}\subset \R^{N}$ such that the translated sequence 
\begin{equation*}
(\tilde{u}_{n}(x), \tilde{v}_{n}(x)):=(u_{n}(x+ \tilde{y}_{n}), v_{n}(x+\tilde{y}_{n}))
\end{equation*}
has a subsequence which converges in $\X_{0}$. Moreover, up to a subsequence, $\{y_{n}\}:=\{\e_{n}\tilde{y}_{n}\}$ is such that $y_{n}\rightarrow y\in M$. 
\end{lem}

\begin{proof}
Since $\langle \J'_{\e_{n}}(u_{n}, v_{n}), (u_{n}, v_{n}) \rangle=0$ and $\J_{\e_{n}}(u_{n}, v_{n})\rightarrow C^{*}$, it is easy to see that $\{(u_{n}, v_{n})\}$ is bounded in $\X_{\e}$. 
Let us observe that $\|(u_{n}, v_{n})\|_{\e_{n}}\nrightarrow 0$ since $C^{*}>0$. Therefore, arguing as in \cite{ASy}, we can find a sequence $\{\tilde{y}_{n}\}\subset \R^{N}$ and constants $R, \gamma>0$ such that
\begin{equation*}
\liminf_{n\rightarrow \infty}\int_{B_{R}(y_{n})} (|u_{n}|^{2}+|v_{n}|^{2}) dx\geq \gamma,
\end{equation*}
and we may assume that
\begin{equation*}
(\tilde{u}_{n}, \tilde{v}_{n})\rightharpoonup (\tilde{u}, \tilde{v}) \mbox{ weakly in } \X_{0},  
\end{equation*}
where $(\tilde{u}_{n}(x), \tilde{v}_{n}(x)):=(u_{n}(x+ \tilde{y}_{n}), v_{n}(x+\tilde{y}_{n}))$ and $(\tilde{u}, \tilde{v})\neq (0,0)$. \\
Let $\{t_{n}\}\subset (0, +\infty)$ be such that $(\hat{u}_{n}, \hat{v}_{n}):=(t_{n}\tilde{u}_{n}, t_{n}\tilde{v}_{n})\in \N_{x_{0}}$, and set $y_{n}:=\e_{n}\tilde{y}_{n}$.  \\
Using the definition of $H$ and $(H3)$ we can see that
\begin{align*}
C^{*}\leq \J_{x_{0}}(\hat{u}_{n}, \hat{v}_{n})&= \frac{t_{n}^{2}}{2} \|(u_{n}, v_{n})\|^{2}_{x_{0}} - \int_{\R^{N}} Q(t_{n} u_{n}, t_{n} v_{n})\, dx \\
&\leq \frac{t_{n}^{2}}{2} \|(u_{n}, v_{n})\|^{2}_{\e_{n}} - \int_{\R^{N}} H(\e_{n} x, t_{n} u_{n}, t_{n} v_{n})\, dx \\
&=\J_{\e_{n}}(t_{n}u_{n}, t_{n}v_{n}) \leq \J_{\e_{n}}(u_{n}, v_{n})= C^{*}+ o_{n}(1),
\end{align*}
which gives $\J_{x_{0}}(\hat{u}_{n}, \hat{v}_{n})\rightarrow C^{*}$. \\
Now, the sequence $\{t_{n}\}$ is bounded since $\{(\tilde{u}_{n}, \tilde{v}_{n})\}$ and $\{(\hat{u}_{n}, \hat{v}_{n})\}$ are bounded in $\X_{0}$, and $(\tilde{u}_{n}, \tilde{v}_{n})\nrightarrow 0$. Therefore, up to a subsequence, $t_{n}\rightarrow t_{0}\geq 0$. Indeed $t_{0}>0$. Otherwise, if $t_{0}=0$, from the boundedness of $\{(\tilde{u}_{n}, \tilde{v}_{n})\}$, we get $(\hat{u}_{n}, \hat{v}_{n})= t_{n}(\tilde{u}_{n}, \tilde{v}_{n}) \rightarrow (0,0)$, that is $\J_{x_{0}}(\hat{u}_{n}, \hat{v}_{n})\rightarrow 0$ in contrast with the fact that $C^{*}>0$. Thus, $t_{0}>0$ and, up to a subsequence, we have $(\hat{u}_{n}, \hat{v}_{n})\rightharpoonup t_{0}(\tilde{u}, \tilde{v})= (\hat{u}, \hat{v})\neq 0$ weakly in $\X_{0}$. 
Hence it holds
\begin{equation*}
\J_{x_{0}}(\hat{u}_{n}, \hat{v}_{n})\rightarrow C^{*} \quad \mbox{ and } \quad (\hat{u}_{n}, \hat{v}_{n})\rightharpoonup (\hat{u}, \hat{v}) \mbox{ weakly in } \X_{0}.
\end{equation*}
From Theorem \ref{prop2.2} we deduce that $(\hat{u}_{n}, \hat{v}_{n})\rightarrow (\hat{u}, \hat{v})$ in $\X_{0}$, that is $(\tilde{u}_{n}, \tilde{v}_{n})\rightarrow (\tilde{u}, \tilde{v})$ in $\X_{0}$. \\
Now we show that $\{y_{n}\}$ has a subsequence, still denoted by itself, such that $y_{n}\rightarrow y\in M$. 
Assume by contradiction that $\{y_{n}\}$ is not bounded, that is there exists a subsequence, still denoted by $\{y_{n}\}$, such that $|y_{n}|\rightarrow +\infty$. 
Since $(u_{n}, v_{n})\in \N_{\e_{n}}$, we can see that
\begin{align*}
&\int_{\R^{N}} |(-\Delta)^{\frac{s}{2}}\tilde{u}_{n}|^{2}+|(-\Delta)^{\frac{s}{2}}\tilde{v}_{n}|^{2}+V(\e_{n}x+y_{n})|\tilde{u}_{n}|^{2}+ W(\e_{n}x+y_{n})|\tilde{v}_{n}|^{2} dx \\
&=\int_{\R^{N}} \tilde{u}_{n}H_{u}(\e_{n} x+y_{n},\tilde{u}_{n} ,\tilde{v}_{n})+\tilde{v}_{n}H_{v}(\e_{n} x+y_{n},\tilde{u}_{n} ,\tilde{v}_{n})\, dx.
\end{align*}
Take $R>0$ such that $\Lambda \subset B_{R}$. Since we may assume that $|y_{n}|>2R$, for any $x\in B_{R/\e_{n}}$ we get $|\e_{n} x+y_{n}|\geq |y_{n}|-|\e_{n} x|>R$.
Then, we deduce that
\begin{align*}
&\int_{\R^{N}} \tilde{u}_{n}H_{u}(\e_{n} x+y_{n},\tilde{u}_{n} ,\tilde{v}_{n})+\tilde{v}_{n}H_{v}(\e_{n} x+y_{n},\tilde{u}_{n} ,\tilde{v}_{n})\, dx \\
&\leq \frac{1}{2} \int_{B_{R/\e_{n}}} V(\e_{n}x+y_{n})|\tilde{u}_{n}|^{2}+ W(\e_{n}x+y_{n})|\tilde{v}_{n}|^{2} dx \\
&+\int_{\R^{N}\setminus B_{R/\e_{n}}} \tilde{u}_{n}H_{u}(\e_{n} x+y_{n},\tilde{u}_{n} ,\tilde{v}_{n})+\tilde{v}_{n}H_{v}(\e_{n} x+y_{n},\tilde{u}_{n} ,\tilde{v}_{n})\, dx \\
&=\frac{1}{2} \int_{B_{R/\e_{n}}} V(\e_{n}x+y_{n})|\tilde{u}_{n}|^{2}+ W(\e_{n}x+y_{n})|\tilde{v}_{n}|^{2} dx+o_{n}(1),
\end{align*}
where we used the strong convergence of $(\tilde{u}_{n}, \tilde{v}_{n})$ and that $|\R^{N}\setminus B_{R/\e_{n}}|\rightarrow 0$ as $n\rightarrow \infty$.
In virtue of $(H3)$ we get
$$
\left(1-\frac{1}{2}\right)\|(\tilde{u}_{n}, \tilde{v}_{n})\|^{2}_{x_{0}}=o_{n}(1),
$$
which is impossible due to $(\tilde{u}_{n}, \tilde{v}_{n})\rightarrow (\tilde{u}, \tilde{v})\neq 0$.
Thus $\{y_{n}\}$ is bounded and, up to a subsequence, we may suppose that $y_{n}\rightarrow y$. If $y\notin \overline{\Lambda}$, then there exists $r>0$ such that $y_{n}\in B_{r/2}(y)\subset \R^{N}\setminus \overline{\Lambda}$ for any $n$ large enough. Reasoning as before, we get a contradiction. Hence $y\in \overline{\Lambda}$. \\
Now, we prove that $y\in M$. Taking into account Lemma \ref{C00}, it is enough to prove that $C(y)=C^{*}$. Assume by contradiction that $C^{*}<C(y)$.
Since $(\hat{u}_{n}, \hat{v}_{n})\rightarrow (\hat{u}, \hat{v})$ strongly in $\X_{0}$, by Fatou's Lemma we have
\begin{align*}
C^{*}<C(y)&= \J_{y}(\hat{u}, \hat{v}) \\
&= \liminf_{n\rightarrow \infty} \Bigl\{ \frac{1}{2}\left(\int_{\R^{N}} |(-\Delta)^{\frac{s}{2}} \hat{u}_{n}|^{2}+ |(-\Delta)^{\frac{s}{2}} \hat{v}_{n}|^{2} dx\right) - \int_{\R^{N}} Q(\hat{u}_{n}, \hat{v}_{n})\, dx \nonumber\\
&+ \frac{1}{2} \int_{\R^{N}} (V(\e_{n}x + y_{n})|\hat{u}_{n}|^{2} + W(\e_{n}x+y_{n})|\hat{v}_{n}|^{2})\, dx\Bigr\} \nonumber \\
&\leq \liminf_{n\rightarrow \infty} \J_{\e_{n}}(t_{n}u_{n}, t_{n}v_{n}) \leq \liminf_{n\rightarrow \infty} \J_{\e_{n}} (u_{n}, v_{n})=C^{*}
\end{align*}
which gives a contradiction. 
\end{proof}

\noindent
Now, we aim to relate the number of positive solutions of \eqref{P''} with the topology of the set $M$.
For this reason, we take $\delta>0$ such that
$$
M_{\delta}=\{x\in \R^{N}: dist(x, M)\leq \delta\}\subset \Lambda,
$$
and we choose $\psi\in C^{\infty}_{0}(\R_{+}, [0, 1])$ a non-increasing function satisfying $\psi(t)=1$ if $0\leq t\leq \frac{\delta}{2}$ and $\psi(t)=0$ if $t\geq \delta$.
For any $y\in M$, we define 
$$
\Psi_{i, \e, y}(x):=\psi(|\e x-y|) w_{i}\left(\frac{\e x-y}{\e}\right) \quad i=1, 2
$$
and denote by $t_{\e}>0$ the unique number such that 
$$
\max_{t\geq 0} \J_{\e}(t \Psi_{1, \e, y}, t \Psi_{2, \e, y})=\J_{\e}(t_{\e} \Psi_{1,\e, y}, t_{\e} \Psi_{2,\e, y}),
$$
where $(w_{1}, w_{2})\in \X_{0}$ is a solution to autonomous system \eqref{P0} with $\xi=x_{0}$, such that $w_{1}, w_{2}>0$ in $\R^{N}$ and $\J_{x_{0}}(w_{1}, w_{2})=C(x_{0})=C^{*}$ (such solution there exists in view of Theorem $3.1$ in \cite{ASy}).

Finally, we consider $\Phi_{\e}: M\rightarrow \N_{\e}$ defined by setting
$$
\Phi_{\e}(y):=(t_{\e} \Psi_{1, \e, y}, t_{\e} \Psi_{2, \e, y}).
$$
Let us prove the following important relationship between $\J_{\e}$ and $M$.
\begin{lem}\label{lemma3.4FS}
The functional $\Phi_{\e}$ satisfies the following limit
\begin{equation}\label{3.2}
\lim_{\e\rightarrow 0} \J_{\e}(\Phi_{\e}(y))=C^{*} \mbox{ uniformly in } y\in M.
\end{equation}
\end{lem}
\begin{proof}
Assume by contradiction that there there exist $\delta_{0}>0$, $\{y_{n}\}\subset M$ and $\e_{n}\rightarrow 0$ such that 
\begin{equation}\label{4.41}
|\J_{\e_{n}}(\Phi_{\e_{n}}(y_{n}))-C^{*}|\geq \delta_{0}.
\end{equation}
We first show that $\lim_{n\rightarrow \infty}t_{\e_{n}}<\infty$.
Let us observe that by using the change of variable $z=\frac{\e_{n}x-y_{n}}{\e_{n}}$, if $z\in B_{\frac{\delta}{\e_{n}}}$, it follows that $\e_{n} z\in B_{\delta}$ and then $\e_{n} z+y_{n}\in B_{\delta}(y_{n})\subset M_{\delta}\subset \Lambda$. 

Then, recalling that $H=Q$ on $\Lambda$ and $\psi(t)=0$ for $t\geq \delta$, we have
\begin{align}\label{HeZou}
\J_{\e_{n}}(\Phi_{\e_{n}}(y_{n}))&=\frac{t_{\e_{n}}^{2}}{2}\int_{\R^{N}} |(-\Delta)^{\frac{s}{2}}(\psi(|\e_{n} z|) w_{1}(z))|^{2}\, dz+\frac{t_{\e_{n}}^{2}}{2}\int_{\R^{N}} |(-\Delta)^{\frac{s}{2}}(\psi(|\e_{n} z|) w_{2}(z))|^{2}\, dz \nonumber\\
&+\frac{t_{\e_{n}}^{2}}{2}\int_{\R^{N}} V(\e_{n} z+y_{n}) (\psi(|\e_{n} z|) w_{1}(z))^{2}\, dz+\frac{t_{\e_{n}}^{2}}{2}\int_{\R^{N}} W(\e_{n} z+y_{n}) (\psi(|\e_{n} z|) w_{2}(z))^{2}\, dz \nonumber\\
&-\int_{\R^{N}} Q(t_{\e_{n}}\psi(|\e_{n} z|)w_{1}(z), t_{\e_{n}}\psi(|\e_{n} z|)w_{2}(z)) \, dz.
\end{align}
Now, let assume that $t_{\e_{n}}\rightarrow \infty$. By the definition of $t_{\e_{n}}$, $(Q1)$ and \eqref{2.1}, we get
\begin{equation}\label{3.9}
\|(\Psi_{1,\e_{n}, y_{n}}, \Psi_{2,\e_{n}, y_{n}})\|^{2}_{\e_{n}}=p t_{\e_{n}}^{p-2}\int_{\R^{N}} Q(\psi(|\e_{n} z|)w_{1}(z), \psi(|\e_{n} z|)w_{2}(z)) \, dz
\end{equation}
Since $\psi=1$ in $B_{\frac{\delta}{2}}$ and $B_{\frac{\delta}{2}}\subset B_{\frac{\delta}{2\e_{n}}}$ for $n$ big enough, and $w_{1}$, $w_{2}$ are continuous and positive in $\R^{N}$ we obtain
\begin{align}\label{3.10}
\|(\Psi_{1,\e_{n}, y_{n}}, \Psi_{2,\e_{n}, y_{n}})\|^{2}_{\e_{n}}\geq p t_{\e_{n}}^{p-2} \int_{B_{\frac{\delta}{2}}} Q(w_{1}(z),w_{2}(z)) \, dz\geq C_{\delta, p} t_{\e_{n}}^{p-2},
\end{align}
for some $C_{\delta, p}>0$.
Taking the limit as $n\rightarrow \infty$ in (\ref{3.10}) we can infer that
$$
\lim_{n\rightarrow \infty} \|(\Psi_{1,\e_{n}, y_{n}}, \Psi_{2,\e_{n}, y_{n}})\|^{2}_{\e_{n}}=\infty
$$
which is a contradiction because of
$$
\lim_{n\rightarrow \infty} \|(\Psi_{1,\e_{n}, y_{n}}, \Psi_{2,\e_{n}, y_{n}})\|^{2}_{\e_{n}}=\|(w_{1}, w_{2})\|^{2}_{x_{0}}\in (0, \infty)
$$
in view of the dominated convergence theorem.\\
Thus, $\{t_{\e_{n}}\}$ is bounded and we can assume that $t_{\e_{n}}\rightarrow t_{0}\geq 0$. Clearly, if $t_{0}=0$, by limitation of $\|(\Psi_{1,\e_{n}, y_{n}}, \Psi_{2,\e_{n}, y_{n}})\|^{2}_{\e_{n}}$, the growth assumptions on $Q$, and (\ref{3.9}), we can deduce that $\|(\Psi_{1,\e_{n}, y_{n}}, \Psi_{2,\e_{n}, y_{n}})\|^{2}_{\e_{n}}\rightarrow 0$, which is impossible. Hence, $t_{0}>0$.

Now, using $(Q2)$ and the dominated convergence theorem we can see that as $n\rightarrow \infty$
$$
\int_{\R^{N}} Q(\Psi_{1, \e_{n}, y_{n}}, \Psi_{2, \e_{n}, y_{n}}) dx\rightarrow \int_{\R^{N}} Q(w_{1}, w_{2}) \, dx.
$$
Then, taking the limit as $n\rightarrow \infty$ in (\ref{3.9})  we obtain
$$
\|(w_{1}, w_{2})\|^{2}_{x_{0}}=p t_{0}^{p-2} \int_{\R^{N}} Q(w_{1}, w_{2}) \, dx.
$$ 
In light of $(w_{1}, w_{2})\in \mathcal{N}_{x_{0}}$, we deduce that $t_{0}=1$. Moreover, from \eqref{HeZou}, we have
$$
\lim_{n\rightarrow \infty} \J_{\e}(\Phi_{\e_{n}}(y_{n}))=\J_{x_{0}}(w_{1}, w_{2})=C^{*},
$$
which contradicts (\ref{4.41}).

\end{proof}

\noindent
At this point, we are in the position to define the barycenter map. We take $\rho=\rho_{\delta}>0$ such that $M_{\delta}\subset B_{\rho}$, and we consider $\varUpsilon: \R^{N}\rightarrow \R^{N}$ given by 
 \begin{equation*}
 \varUpsilon(x)=
 \left\{
 \begin{array}{ll}
 x &\mbox{ if } |x|<\rho \\
 \frac{\rho x}{|x|} &\mbox{ if } |x|\geq \rho.
  \end{array}
 \right.
 \end{equation*}
We define the barycenter map $\beta_{\e}: \N_{\e}\rightarrow \R^{N}$ as follows
\begin{align*}
\beta_{\e}(u, v)=\frac{\int_{\R^{N}} \varUpsilon(\e x)(u^{2}(x)+v^{2}(x)) dx}{\int_{\R^{N}} (u^{2}(x)+v^{2}(x)) dx}.
\end{align*}

\begin{lem}\label{lemma3.5FS}
The functional $\Phi_{\e}$ satisfies the following limit
\begin{equation}\label{3.3}
\lim_{\e \rightarrow 0} \beta_{\e}(\Phi_{\e}(y))=y \mbox{ uniformly in } y\in M.
\end{equation}
\end{lem}
\begin{proof}
Suppose by contradiction that there exist $\delta_{0}>0$, $\{y_{n}\}\subset M$ and $\e_{n}\rightarrow 0$ such that 
\begin{equation}\label{4.4}
|\beta_{\e_{n}}(\Phi_{\e_{n}}(y_{n}))-y_{n}|\geq \delta_{0}.
\end{equation}
Using the definitions of $\Phi_{\e_{n}}(y_{n})$, $\beta_{\e_{n}}$, $\eta$ and the change of variable $z=\frac{\e_{n} x-y_{n}}{\e_{n}}$, we can see that 
$$
\beta_{\e_{n}}(\Phi_{\e_{n}}(y_{n}))=y_{n}+\frac{\int_{\R^{N}}[\Upsilon(\e_{n}z+y_{n})-y_{n}] |\eta(|\e_{n}z|)|^{2} (|w_{1}(z)|^{2}+|w_{2}(z)|^{2}) \, dz}{\int_{\R^{N}} |\eta(|\e_{n}z|)|^{2} (|w_{1}(z)|^{2}+|w_{2}(z)|^{2})\, dz}.
$$
Taking into account that $\{y_{n}\}\subset M\subset B_{\rho}$ and applying the dominated convergence theorem we can infer that 
$$
|\beta_{\e_{n}}(\Phi_{\e_{n}}(y_{n}))-y_{n}|=o_{n}(1)
$$
which contradicts (\ref{4.4}).
\end{proof}

\noindent
We now introduce a subset $\widetilde{\N}_{\e}$ of $\N_{\e}$ by taking a function $h:\R_{+}\rightarrow \R_{+}$ such that $h(\e)\rightarrow 0$ as $\e \rightarrow 0$, and setting
$$
\widetilde{\N}_{\e}=\{(u, v)\in \N_{\e}: \J_{\e}(u, v)\leq C^{*}+h(\e)\}.
$$
Fixed $y\in M$, we conclude from Lemma \ref{lemma3.4FS} that $h(\e)=|\J_{\e}(\Phi_{\e}(y))-C^{*}|\rightarrow 0$ as $\e \rightarrow 0$. Hence, $\Phi_{\e}(y)\in \widetilde{\N}_{\e}$ and $\widetilde{\N}_{\e}\neq \emptyset$ for any $\e>0$ small. Moreover, we have the following interesting relation between $\widetilde{\N}_{\e}$ and $\beta_{\e}$.
\begin{lem}\label{lemma3.7FS}
For any $\delta>0$, there holds that
$$
\lim_{\e \rightarrow 0} \sup_{(u, v)\in \widetilde{\mathcal{N}}_{\e}} dist(\beta_{\e}(u, v), M_{\delta})=0.
$$
\end{lem}

\begin{proof}
Let $\e_{n}\rightarrow 0$ as $n\rightarrow \infty$. For any $n\in \mathbb{N}$, there exists $(u_{n}, v_{n})\in \widetilde{\N}_{\e_{n}}$ such that
$$
\sup_{(u, v)\in \widetilde{\N}_{\e_{n}}} \inf_{y\in M_{\delta}}|\beta_{\e_{n}}(u, v)-y|=\inf_{y\in M_{\delta}}|\beta_{\e_{n}}(u_{n}, v_{n})-y|+o_{n}(1).
$$
Therefore, it is suffices to find a sequence $\{y_{n}\}\subset M_{\delta}$ such that 
\begin{equation}\label{3.13}
\lim_{n\rightarrow \infty} |\beta_{\e_{n}}(u_{n}, v_{n})-y_{n}|=0.
\end{equation}
We note that $\{(u_{n}, v_{n})\}\subset  \widetilde{\N}_{\e_{n}}\subset  \N_{\e_{n}}$, from which we obtain that
$$
C^{*}\leq c_{\e_{n}}\leq \J_{\e_{n}}(u_{n}, v_{n})\leq C^{*}+h(\e_{n}).
$$
This yields that $\J_{\e_{n}}(u_{n}, v_{n})\rightarrow C^{*}$. Using Lemma \ref{lem3.1}, there exists $\{\tilde{y}_{n}\}\subset \R^{N}$ such that $y_{n}=\e_{n}\tilde{y}_{n}\in M_{\delta}$ for $n$ sufficiently large. Setting $(\tilde{u}_{n}(x), \tilde{v}_{n}(x))=(u_{n}(\cdot+\tilde{y}_{n}), v_{n}(\cdot+\tilde{y}_{n}))$, we can see that
$$
\beta_{\e_{n}}(u_{n}, v_{n})=y_{n}+\frac{\int_{\R^{N}}[\Upsilon(\e_{n}x+y_{n})-y_{n}] (\tilde{u}_{n}^{2}+\tilde{v}_{n}^{2}) \, dx}{\int_{\R^{N}} (\tilde{u}_{n}^{2}+\tilde{v}_{n}^{2})\, dx}.
$$
Since $(\tilde{u}_{n}, \tilde{v}_{n})\rightarrow (u, v)$ in $\X_{0}$ and $\e_{n}x+y_{n}\rightarrow y\in M_{\delta}$, we deduce that $\beta_{\e_{n}}(u_{n}, v_{n})=y_{n}+o_{n}(1)$, that is (\ref{3.13}) holds.

\end{proof}

\noindent
Now, we are ready to present the proof of the multiplicity result for \eqref{P''}.
\begin{thm}\label{thm3.1AFF}
For any $\delta>0$ satisfying $M_{\delta}\subset M$, there exists $\e_{\delta}>0$ such that for any $\e\in (0, \e_{\delta})$, problem \eqref{P''} has at least $cat_{M_{\delta}}(M)$ positive solutions.
\end{thm}
\begin{proof}
Given $\delta>0$ such that $M_{\delta}\subset \Lambda$, we can apply Lemma \ref{lemma3.4FS}, Lemma \ref{lemma3.5FS} and Lemma \ref{lemma3.7FS} to find $\e_{\delta}>0$ such that for any $\e\in (0, \e_{\delta})$, the following diagram
$$
M \stackrel{\Phi_{\e}}{\rightarrow}  \widetilde{\N}_{\e} \stackrel{\beta_{\e}}{\rightarrow} M_{\delta}
$$
is well-defined and $\beta_{\e}\circ \Phi_{\e}$ is  homotopically equivalent to the embedding $\iota: M\rightarrow M_{\delta}$. 
Using the definition of $\widetilde{\N}_{\e}$ and taking $\e_{\delta}$ sufficiently small, we may assume that $\J_{\e}$ fulfills the Palais-Smale condition in $\widetilde{\N}_{\e}$ (see Proposition \ref{prop1}). Therefore, standard Ljusternik-Schnirelmann theory \cite{W} provides at least $cat_{\widetilde{\N}_{\e}}(\widetilde{\N}_{\e})$ critical points $(u_{i}, v_{i}):=(u^{i}_{\e}, v^{i}_{\e})$ of $\J_{\e}$ restricted to $\N_{\e}$. Using the arguments in \cite{BC}, we know that $cat_{\widetilde{\N}_{\e}}(\widetilde{\N}_{\e})\geq cat_{M_{\delta}}(M)$. 
Then, arguing as in the proof of Proposition \ref{prop1}, we can see that $(u_{i}, v_{i})$ is also a critical point of the unconstrained functional and therefore a solution of problem \eqref{P''}.

\end{proof}

\section{proof of theorem \ref{thm1}}

In this last section we provide the proof of our main result.
\begin{proof}
Take $\delta>0$ sufficiently small such that $M_\delta \subset \Lambda$. We begin by proving that there exists $\tilde{\e}_{\delta}>0$ such that for any $\e \in (0, \tilde{\e}_{\delta})$ and any solution $u_{\e} \in \widetilde{\N}_{\e}$ of \eqref{P''} it holds 
\begin{equation}\label{infty}
\|(u_{\e}, v_{\e})\|_{L^{\infty}(\R^{N}\setminus \Lambda_{\e})}<a. 
\end{equation}
Assume by contradiction that there exist $\e_{n}\rightarrow 0$, $(u_{\e_{n}}, v_{\e_{n}})\in \widetilde{\mathcal{N}}_{\e_{n}}$ such that $\J'_{\e_{n}}(u_{\e_{n}}, v_{\e_{n}})=0$ and $\|(u_{\e_{n}}, v_{\e_{n}})\|_{L^{\infty}(\R^{N}\setminus \Lambda_{\e_{n}})}\geq a$. 
Since $\J_{\e_{n}}(u_{\e_{n}}, v_{\e_{n}}) \leq C^{*} + h(\e_{n})$ and $h(\e_{n})\rightarrow 0$, we can argue as in the  first part of the proof of Lemma \ref{lem3.1}, to deduce that $\J_{\e_{n}}(u_{\e_{n}}, v_{\e_{n}})\rightarrow C^{*}$.
Then, invoking Lemma \ref{lem3.1}, we can find $\{\tilde{y}_{n}\}\subset \R^{N}$ such that $\e_{n}\tilde{y}_{n}\rightarrow y \in M$. \\
Now, if we choose $r>0$ such that $B_{r}(y)\subset B_{2r}(y)\subset \Lambda$, we have $B_{\frac{r}{\e_{n}}}(\frac{y}{\e_{n}})\subset \Lambda_{\e_{n}}$. In particular, for any $z\in B_{\frac{r}{\e_{n}}}(\tilde{y}_{n})$ there holds
\begin{equation*}
\left|z - \frac{y}{\e_{n}}\right| \leq |z-\tilde{y}_{n}|+ \left|\tilde{y}_{n} - \frac{y}{\e_{n}}\right|<\frac{2r}{\e_{n}}\, \mbox{ for } n \mbox{ sufficiently large. }
\end{equation*}
Therefore $\R^{N}\setminus \Lambda_{\e_{n}}\subset \R^{N} \setminus B_{\frac{r}{\e_{n}}}(\tilde{y}_{n})$ for any $n$ big enough.  \\
Now, let us denote by $(\tilde{u}_{n}(x), \tilde{v}_{n}(x))=(u_{\e_{n}}(x+ \tilde{y}_{n}), v_{\e_{n}}(x+ \tilde{y}_{n}))$ and $\tilde{z}_{n}=\tilde{u}_{n}+\tilde{v}_{n}\geq 0$. 
Using $(H3)$, the definition of $H$ and the growth conditions on $Q$, we can see that $\tilde{z}_{n}$ satisfies
\begin{equation}\label{LMS}
(-\Delta)^{s} \tilde{z}_{n}+\alpha \tilde{z}_{n}\leq g_{n} \mbox{ in } \R^{N}, 
\end{equation}
where $\alpha=\min\{V(x_{0}), W(x_{0})\}$ and $g_{n}$ is such that $|g_{n}|\leq \xi \tilde{z}_{n}+C_{\xi} \tilde{z}_{n}^{p-1}$, with $\xi>0$ fixed. \\ 
Then, for $\beta>0$ and $L>1$, we take $\tilde{z}_{n}\tilde{z}_{L, n}^{2(\beta-1)}$, where $\tilde{z}_{L, n}=\min\{\tilde{z}_{n}, L\}$, as test function in \eqref{LMS}, and arguing as in the proof of Lemma $6.1$ in \cite{AI} (see also Lemma $5.1$ in \cite{A5}) and observing that $\{\tilde{z}_{n}\}$ is bounded in $L^{2^{*}_{s}}(\R^{N})$ (since $\{(u_{\e_{n}}, v_{\e_{n}})\}$ is bounded in $\X_{\e_{n}}$), we can use a Moser iteration scheme to deduce that $\tilde{z}_{n}\in L^{\infty}(\R^{N})$ and there exists a constant $K>0$ such that  
$$
\|\tilde{z}_{n}\|_{L^{\infty}(\R^{N})}\leq K \mbox{ for any } n\in \mathbb{N}.
$$
Consequently, $\{\tilde{u}_{n}\}$ and $\{\tilde{v}_{n}\}$ are bounded in $L^{\infty}(\R^{N})$, and by interpolation,
$\tilde{u}_{n}\rightarrow u$  and $\tilde{v}_{n}\rightarrow v$ in $L^{q}(\R^{N})$ for any $q\in (2, \infty)$, for some $u, v\in L^{q}(\R^{N})$ for any $q\in (2, \infty)$.\\
Then, from the growth conditions on $Q$, we also have the following relations of limit in $L^{q}(\R^{N})$ for any $q\in (2, \infty)$:
$$
H_{u}(\e_{n}x+\e_{n}\tilde{y}_{n}, \tilde{u}_{n}, \tilde{v}_{n})\rightarrow Q_{u}(u, v)
$$
and 
$$
H_{v}(\e_{n}x+\e_{n}\tilde{y}_{n}, \tilde{u}_{n}, \tilde{v}_{n})\rightarrow Q_{v}(u, v).
 $$ 
Since $\tilde{z}_{n}$ satisfies 
$$
(-\Delta)^{s} \tilde{z}_{n}+\tilde{z}_{n}=\xi_{n} \mbox{ in } \R^{N},
$$
where 
\begin{align*}
\xi_{n}&:=H_{u}(\e_{n}x+\e_{n}\tilde{y}_{n}, \tilde{u}_{n}, \tilde{v}_{n})+H_{v}(\e_{n}x+\e_{n}\tilde{y}_{n}, \tilde{u}_{n}, \tilde{v}_{n}) \\
&-V(\e_{n}x+\e_{n}\tilde{y}_{n}) \tilde{u}_{n}-W(\e_{n}x+\e_{n}\tilde{y}_{n})\tilde{v}_{n}+\tilde{z}_{n},
\end{align*}
we have that
$$
\xi_{n}\rightarrow Q_{u}(u, v)+Q_{v}(u, v)-V(y)u-W(y) v+z \mbox{ in } L^{q}(\R^{N})
$$
for any $q\in [2, \infty)$, and we can find $K_{1}>0$ such that 
$$
\|\xi_{n}\|_{L^{\infty}(\R^{N})}\leq K_{1} \mbox{ for any } n\in \mathbb{N}.
$$
Hence $\tilde{z}_{n}(x)=(\mathcal{K}*\xi_{n})(x)=\int_{\R^{N}} \mathcal{K}(x-t) \xi_{n}(t) \, dt$, where $\mathcal{K}$ is the Bessel kernel which satisfies the following properties (see \cite{FQT}):
\begin{compactenum}[$(i)$]
\item $\mathcal{K}$ is positive, radially symmetric and smooth in $\R^{N}\setminus \{0\}$,
\item there is $C>0$ such that $\mathcal{K}(x)\leq \frac{C}{|x|^{N+2s}}$ for any $x\in \R^{N}\setminus \{0\}$,
\item $\mathcal{K}\in L^{q}(\R^{N})$ for any $q\in [1, \frac{N}{N-2s})$.
\end{compactenum}
Then, arguing as in Lemma $2.6$ in \cite{AM}, we can see that 
\begin{equation}\label{AM3}
\tilde{z}_{n}(x)\rightarrow 0 \mbox{ as } |x|\rightarrow \infty
\end{equation}
uniformly in $n\in \mathbb{N}$.\\
Therefore, there exists $R>0$ such that 
$$
|(\tilde{u}_{n}(x), \tilde{v}_{n}(x))|<a \quad \mbox{ for all } |x|\geq R \mbox{ and } n\in \mathbb{N},
$$
from which
$$
|(u_{\e_{n}}(x), v_{\e_{n}}(x))|<a \quad \mbox{ for all } x\in \R^{N}\setminus B_{R}(\tilde{y}_{n})  \mbox{ and } n\in \mathbb{N}.
$$
On the other hand, there exists $\nu \in \mathbb{N}$ such that for any $n\geq \nu$ and $\frac{r}{\e_{n}}>R$, it holds 
 $$
 \R^{N}\setminus \Lambda_{\e_{n}}\subset \R^{N} \setminus B_{\frac{r}{\e_{n}}}(\tilde{y}_{n})\subset \R^{N}\setminus B_{R}(\tilde{y}_{n}),
 $$
which gives $|(u_{\e_{n}}(x), v_{\e_{n}}(x))|<a$ for any $x\in \R^{N}\setminus \Lambda_{\e_{n}}$, that is a contradiction. \\
Now, let $\bar{\e}_{\delta}$ be given by Theorem \ref{thm3.1AFF} and take $\e_{\delta}= \min \{\tilde{\e}_{\delta}, \bar{\e}_{\delta}\}$. Fix $\e \in (0, \e_{\delta})$. By Theorem \ref{thm3.1AFF} we know that problem \eqref{P''} admits $cat_{M_{\delta}}(M)$ nontrivial solutions $(u_{\e}, v_{\e})$. Since $(u_{\e}, v_{\e})\in \widetilde{\mathcal{N}}_{\e}$ satisfies \eqref{infty}, by the definitions of $H$ and $\hat{Q}$ it follows that $(u_{\e}, v_{\e})$ is a solution of \eqref{P'}. In light of $(Q6)$ and the maximum principle for the fractional Laplacian \cite{CabSir}, we can infer that $u_{\e}, v_{\e}>0$ in $\R^{N}$.

Now, we study the behavior of maximum points of solutions to \eqref{P}.\\
Let $\e_{n}\rightarrow 0$ and take $\{(u_{\e_{n}}, v_{\e_{n}})\}\subset \X_{\e_{n}}$ be a sequence of solutions to \eqref{P''} as above.
Using the definition of $H$ and $(Q2)$ we can see that there exists $\bar{a}\in (0, a)$ sufficiently small such that
\begin{equation}\label{HZnew}
uH_{u}(\e_{n} x, u, v)+vH_{v}(\e_{n} x, u, v)\leq \frac{\alpha}{2}(u^{2}+v^{2}) \mbox{ for all } x\in \R^{N}, |(u, v)|\leq \bar{a}.
\end{equation}
Arguing as before, we can find $R>0$ such that 
\begin{equation}\label{HZnnew}
\|(u_{\e_{n}}, v_{\e_{n}})\|_{L^{\infty}(B_{R}^{c}(\tilde{y}_{n}))}< \bar{a}.
\end{equation}
Up to a subsequence, we may assume that 
\begin{equation}\label{TV1}
\|(u_{\e_{n}}, v_{\e_{n}})\|_{L^{\infty}(B_{R}(\tilde{y}_{n}))}\geq \bar{a}.
\end{equation}
Indeed, if this case does not occur, we deduce that $\|(u_{\e_{n}}, v_{\e_{n}})\|_{L^{\infty}(\R^{N})}< \bar{a}$ and using the facts that $\langle \J'_{\e_{n}}(u_{\e_{n}}, v_{\e_{n}}),(u_{\e_{n}}, v_{\e_{n}}) \rangle=0$ and \eqref{HZnew} we get
\begin{align*}
\|(u_{\e_{n}}, v_{\e_{n}}) \|^{2}_{\e_{n}}=\int_{\R^{N}} u_{\e_{n}} H_{u}(\e_{n} x, u_{\e_{n}}, v_{\e_{n}})+v_{\e_{n}} H_{v}(\e_{n} x, u_{\e_{n}}, v_{\e_{n}}) \, dx\leq \frac{\alpha}{2}\int_{\R^{N}} (u_{\e_{n}}^{2}+ v_{\e_{n}}^{2})\, dx,
\end{align*}
which implies that $\| (u_{\e_{n}}, v_{\e_{n}})\|_{\e_{n}}\rightarrow 0$ as $n\rightarrow \infty$, that is a contradiction. Then \eqref{TV1} holds.
Therefore, if we denote by $x_{n}$ and $\bar{x}_{n}$ the  maximum points of $u_{\e_{n}}$ and $v_{\e_{n}}$ respectively, it follows from \eqref{HZnnew} and \eqref{TV1} that $x_{n}=\tilde{y}_{n}+p_{n}$ and $\bar{x}_{n}=\tilde{y}_{n}+q_{n}$ for some $p_{n}, q_{n}\in B_{R}$. \\
Set $\hat{u}_{n}(x)=u_{\e_{n}}(x/\e_{n})$ and $\hat{v}_{n}(x)=v_{\e_{n}}(x/\e_{n})$.
Then $\hat{u}_{n}$ and $\hat{v}_{n}$ are solutions to \eqref{P} with maximum points $P_{n}:=\e_{n}\tilde{y}_{n}+\e_{n}p_{n}$ and $Q_{n}:=\e_{n}\tilde{y}_{n}+\e_{n}q_{n}$ respectively. Since $|p_{n}|, |q_{n}|<R$ for all $n\in \mathbb{N}$ and $\e_{n}\tilde{y}_{n}\rightarrow y\in M$ we can infer that $P_{n}, Q_{n} \rightarrow y$. 
By using Lemma \ref{C0} we obtain
$$
\lim_{n\rightarrow \infty} C(P_{n})=\lim_{n\rightarrow \infty} C(Q_{n})=C(y)=C^{*}=C(x_{0}).
$$
Finally, we study the decay  properties of $(\hat{u}_{n}, \hat{v}_{n})$ and we prove that \eqref{DEuv} holds. 

Let us define $\tilde{z}_{n}(x)=\tilde{u}_{n}(x)+\tilde{v}_{n}(x)$.
By (\ref{AM3}) it follows that $\tilde{z}_{n}\rightarrow 0$ as $|x|\rightarrow \infty$ uniformly in $n$.
We recall that $(Q2)$ gives
$$
|H_{u}|+|H_{v}|= o(|(u, v)|) \mbox{ as } |(u, v)|\rightarrow 0.
$$
Then, setting $V_{n}:=V(\e_{n}x+\e_{n}\tilde{y}_{n})$, $W_{n}:=W(\e_{n}x+\e_{n}\tilde{y}_{n})$, and using $(H3)$, $\alpha=\min\{V(x_{0}), W(x_{0})\}$, $\sqrt{x^{2}+y^{2}}\leq x+y$ for any $x, y\geq 0$, we can find $R_{1}>0$ sufficiently large such that
\begin{align}\label{HZ3}
&(-\Delta)^{s} \tilde{z}_{n}+\frac{\alpha}{2} \tilde{z}_{n}\nonumber \\
&=(-\Delta)^{s} \tilde{u}_{n}+(-\Delta)^{s} \tilde{v}_{n}+V_{n} \tilde{u}_{n}+W_{n} \tilde{v}_{n}-\left(V_{n} \tilde{u}_{n}+W_{n} \tilde{v}_{n}-\frac{\alpha}{2} \tilde{z}_{n} \right) \nonumber \\
&=H_{u}(\e_{n}x+\e_{n}\tilde{y}_{n}, \tilde{u}_{n}, \tilde{v}_{n})+H_{v}(\e_{n}x+\e_{n}\tilde{y}_{n}, \tilde{u}_{n}, \tilde{v}_{n}) \nonumber \\
&-\left(V_{n} \tilde{u}_{n}+W_{n} \tilde{v}_{n}-\frac{\alpha}{2} \tilde{z}_{n} \right) \nonumber \\
&\leq H_{u}(\e_{n}x+\e_{n}\tilde{y}_{n}, \tilde{u}_{n}, \tilde{v}_{n})+H_{v}(\e_{n}x+\e_{n}\tilde{y}_{n}, \tilde{u}_{n}, \tilde{v}_{n}) -\frac{\alpha}{2} \tilde{z}_{n} \nonumber \\
&\leq H_{u}(\e_{n}x+\e_{n}\tilde{y}_{n}, \tilde{u}_{n}, \tilde{v}_{n})+H_{v}(\e_{n}x+\e_{n}\tilde{y}_{n}, \tilde{u}_{n}, \tilde{v}_{n}) -\frac{\alpha}{2} |(\tilde{u}_{n}, \tilde{v}_{n})|  \nonumber \\
&\leq 0 \mbox{ in } \R^{N}\setminus B_{R_{1}}. 
\end{align}
On the other hand, by Lemma $4.3$ in \cite{FQT}, we know that there exists a function $w$ such that 
\begin{align}\label{HZ1}
0<w(x)\leq \frac{C}{1+|x|^{N+2s}},
\end{align}
and
\begin{align}\label{HZ2}
(-\Delta)^{s} w+\frac{\alpha}{2}w\geq 0 \mbox{ in } \R^{N}\setminus B_{R_{2}} 
\end{align}
for some suitable $R_{2}>0$.
Take $R_{3}=\max\{R_{1}, R_{2}\}$ and we set 
\begin{align}\label{HZ4}
c=\inf_{B_{R_{3}}} w>0 \mbox{ and } \tilde{w}_{n}=(b+1)w-c\tilde{z}_{n},
\end{align}
where $b=\sup_{n\in \mathbb{N}} \|\tilde{z}_{n}\|_{L^{\infty}(\R^{N})}<\infty$. 
In what follows we prove that 
\begin{equation}\label{HZ5}
\tilde{w}_{n}\geq 0 \mbox{ in } \R^{N}.
\end{equation}
First of all, we can note that
\begin{align}
&\tilde{w}_{n}\geq bc+w-bc>0 \mbox{ in } B_{R_{3}} \label{HZ0},\\
&(-\Delta)^{s} \tilde{w}_{n}+\frac{\alpha}{2}\tilde{w}_{n}\geq 0 \mbox{ in } \R^{N}\setminus B_{R_{3}} \label{HZ00}.
\end{align}
Now we argue by contradiction and assume that there exists a sequence $\{\bar{x}_{j, n}\}\subset \R^{N}$ such that 
\begin{align}\label{HZ6}
\inf_{x\in \R^{N}} \tilde{w}_{n}(x)=\lim_{j\rightarrow \infty} \tilde{w}_{n}(\bar{x}_{j, n})<0. 
\end{align}
Using (\ref{AM3}), $\eqref{HZ1}$ and the definition of $\tilde{w}_{n}$, we know that $\tilde{w}_{n}(x)\rightarrow 0$ as $|x|\rightarrow \infty$, uniformly in $n\in \mathbb{N}$. Therefore $\{\bar{x}_{j, n}\}$ is bounded, and, up to subsequence, we may assume that there exists $\bar{x}_{n}\in \R^{N}$ such that $\bar{x}_{j, n}\rightarrow \bar{x}_{n}$ as $j\rightarrow \infty$. 
Thanks to (\ref{HZ6}) we can see that
\begin{align}\label{HZ7}
\inf_{x\in \R^{N}} \tilde{w}_{n}(x)= \tilde{w}_{n}(\bar{x}_{n})<0.
\end{align}
From the minimality of $\bar{x}_{n}$ and the integral representation of the fractional Laplacian \cite{DPV}, we can see that 
\begin{align}\label{HZ8}
(-\Delta)^{s}\tilde{w}_{n}(\bar{x}_{n})=\frac{C_{N, s}}{2} \int_{\R^{N}} \frac{2\tilde{w}_{n}(\bar{x}_{n})-\tilde{w}_{n}(\bar{x}_{n}+\xi)-\tilde{w}_{n}(\bar{x}_{n}-\xi)}{|\xi|^{N+2s}} d\xi\leq 0.
\end{align}
In view of (\ref{HZ0}) and (\ref{HZ6}), we can infer that $\bar{x}_{n}\in \R^{N}\setminus B_{R_{3}}$.
This fact combined with (\ref{HZ7}) and (\ref{HZ8}) yields 
$$
(-\Delta)^{s} \tilde{w}_{n}(\bar{x}_{n})+\frac{\alpha}{2}\tilde{w}_{n}(\bar{x}_{n})<0,
$$
which gives a contradiction due to (\ref{HZ00}).
Accordingly, (\ref{HZ5}) holds true, and using (\ref{HZ1}) we have
\begin{align*}
\tilde{z}_{n}(x)\leq \frac{\tilde{C}}{1+|x|^{N+2s}} \mbox{ for all } x\in \R^{N}, n\in \mathbb{N},
\end{align*}
for some constant $\tilde{C}>0$. 
Recalling the definition of $\tilde{z}_{n}$ we can deduce that  
\begin{align*}
\hat{u}_{n}(x)&=u_{\e_{n}}\left(\frac{x}{\e_{n}}\right)=\tilde{u}_{n}\left(\frac{x}{\e_{n}}-\tilde{y}_{n}\right) \\
&\leq \frac{\tilde{C}}{1+|\frac{x}{\e_{n}}-\tilde{y}_{n}|^{N+2s}} \\
&=\frac{\tilde{C} \e_{n}^{N+2s}}{\e_{n}^{N+2s}+|x- \e_{n} \tilde{y}_{n}|^{N+2s}} \\
&\leq \frac{\tilde{C} \e_{n}^{N+2s}}{\e_{n}^{N+2s}+|x-P_{n}|^{N+2s}} \quad \forall x\in \R^{N}.
\end{align*}
In a similar manner we can obtain the estimate for $\hat{v}_{n}$. This ends the proof of Theorem \ref{thm1}. 
\end{proof}

\end{document}